\documentclass[12pt]{article}
\usepackage[utf8]{inputenc}
\usepackage[T1]{fontenc}
\usepackage{lmodern}
\usepackage{amsmath,amsthm,amsfonts,amssymb, stmaryrd,bbm,mathrsfs, dsfont}
\usepackage{mathtools}
\usepackage{graphicx}
\usepackage{enumitem}
\usepackage{url}
\usepackage{appendix}
\usepackage{textcomp}
\usepackage[cal=euler]{mathalfa}
\def\llbracket{[\hspace{-.10em} [ }
\def\rrbracket{ ] \hspace{-.10em}]}
\usepackage[letterpaper,top=3cm,bottom=3cm,left=2.5cm,right=2.5cm,marginparwidth=1.5cm]{geometry}
\usepackage[font={small,sf}, labelfont={sf,bf}, margin=1cm]{caption}
\usepackage{csquotes}

\usepackage[english]{babel}

\usepackage[usenames,dvipsnames]{xcolor}
\definecolor{darkblue}{rgb}{0.0, 0.2, 0.6}
\usepackage[pdftitle={},
  pdfauthor={},
colorlinks=true,linkcolor=RoyalBlue,urlcolor=MidnightBlue,citecolor=darkblue,bookmarks=true,bookmarksopen=true,bookmarksopenlevel=2,unicode=true,linktocpage]{hyperref}

\numberwithin{equation}{section}

\usepackage[capitalise,noabbrev]{cleveref}
\usepackage{thmtools}

\usepackage[usenames,dvipsnames]{xcolor}
\usepackage[labelfont={bf,sf}, textfont={sf}]{caption}
\usepackage{subfigure}
\theoremstyle{plain}
\newtheorem{thm}{Theorem}[section]
\newtheorem*{thm*}{Theorem}
\newtheorem{prop}[thm]{Proposition}
\newtheorem*{prop*}{Proposition}
\newtheorem{defn}[thm]{Definition}
\newtheorem*{Def*}{Definition}
\newtheorem{lem}[thm]{Lemma}
\newtheorem*{Cor*}{Corollary}
\newtheorem{Cor}[thm]{Corollary}

\newtheorem{Assump}[thm]{Assumption}

\theoremstyle{definition}
\newtheorem{Ex}[thm]{Example}
\newtheorem*{Ex*}{Example}
\newtheorem{rem}[thm]{Remark}
\newtheorem*{rem*}{Remarques}

\newcommand{\D}{\mathrm{d}}
\newcommand{\Bd}[1]{\boldsymbol{#1}}

\newcommand{\ord}{\mathrm{ord}}

\usepackage{array}
\newcolumntype{L}[1]{>{\raggedright\arraybackslash}p{#1}}
\DeclareUnicodeCharacter{0308}{\"{}}
\title{Critical Self-Similar Markov Trees}
\author{Nicolas Curien\thanks{Universit\'e Paris-Saclay, \url{nicolas.curien@universite-paris-saclay.fr}},   Xingjian Hu \thanks{Fudan University, \url{xjhu22@m.fudan.edu.cn}} and  Dongjian Qian \thanks{Fudan University, \url{djqian22@m.fudan.edu.cn}}}
\date{}

\begin{document}

\maketitle

\begin{abstract}
Recently introduced and studied in \cite{BertoinJean2024SMta}, a self-similar Markov tree (ssMt) is a random decorated tree that vastly generalises the fragmentation tree. We study here the \emph{critical case} that was left aside in \cite{BertoinJean2024SMta}. Borrowing techniques from branching random walk, in particular the recent result of Aïdékon--Hu--Shi \cite{AidekonElie2024Bodt}, we can complete the picture by constructing critical ssMt, computing their fractal dimension and studying their associated harmonic and length measures using spinal decomposition.
\end{abstract}

\section{Introduction}
Self-similar Markov trees have  recently been introduced in \cite{BertoinJean2024SMta}. They can be seen as generalisations of the fragmentation trees of Haas \& Miermont \cite{HaasBenedicte2004TGoS}, or as the genealogical trees underlying the Lamperti transformation of branching L\'evy processes \cite{BertoinJean2019IRPM}. They encompass and unify various models including  Aldous' famous Brownian Continuum Random tree \cite{AldousDavid1991TCRT} and its stable generalisations \cite{DuquesneThomas2002RtLp} as well as more exotic trees such as the Brownian Cactus \cite{CurienNicolas2013TBcI} sitting inside the Brownian sphere \cite{LeGallJean-Francois2020GPIB, bertoin2018random}, or the scaling limits of peeling or parking trees \cite{BertoinJean2018Misg,contat2025universality}. They are conjectured to describe scaling limits of multi-type Bienaym\'e--Galton--Watson trees, exactly as Lamperti's self-similar Markov processes describe the scaling limits of positive Markov chains on the integers. Their central role in the theory of random tree sparked a recent interest, see e.g. \cite{bertoin2026local, curien2025growing} for their intrinsic studies or \cite{aidekon2022growth, LeGallJean-Francois2020GPIB, bertoin2018random, DaSilvaWilliam2025GBce} for their connections with various models of random planar geometry.

Formally, a self-similar Markov tree (ssMt) is a family of laws $( \mathbb{Q}_x : x >0)$ of decorated continuous random $\mathbb{R}$-trees $\mathtt{T} = (T, d_T, \rho, g)$ where $\rho$ is the root of the tree and $g : T \to \mathbb{R}_+$ is a real decoration which is upper semi-continuous (usc) on $T$ and positive on its skeleton. Under the law $\mathbb{Q}_x$, the random decorated tree $\mathtt{T}$ starts from the initial decoration  $g(\rho)=x$ and enjoys the following two eponymic properties:
\begin{enumerate}
    \item \textbf{Markov property.} For each $h>0$, conditioned on the subtree truncated at height $h$, the decorated subtrees above $h$ are independent of each other and have law $\mathbb{Q}_y$ if the decoration at its root is $y$.
     \item \textbf{Self-similarity.} There exists $\alpha > 0$,  the \textbf{self-similar index}, such that for each $x>0$, the tree  $(T, d_{T}, \rho, g)$ under $\mathbb{Q}_x$ has the same law as $(T,x^{\alpha}d_T,\rho,x\cdot g) $ under $\mathbb{Q}_1$.
\end{enumerate}
Following \cite[Section 2]{BertoinJean2024SMta}, a self-similar Markov tree can be described by its characteristic quadruplet $(\sigma^2, \mathrm{a}, \boldsymbol{\Lambda}; \alpha)$, which through a Lamperti transformation encapsulates the law of their underlying branching L\'evy processes. In particular, the generalised L\'evy measure $\boldsymbol{\Lambda}$ is a (possibly infinite) measure on the space $$ \mathcal{S} = \{ \mathbf{u}=(u_0,u_1, ...) : u_0  \in \mathbb{R} \mbox{ and }  u_1\geq u_2 \geq ... \in \mathbb{R} \cup \{ - \infty\}\}$$ which describes the splitting rules: Informally, the ssMt is the genealogical tree of a system of individuals evolving  independently of each other, and where an individual of decoration $x >0$ sees its decoration instantaneously moved to $x\cdot\mathrm{e}^{y_0}$ while giving rise to a  family of new individuals with decorations $x \cdot \mathrm{e}^{y_1}, x \cdot \mathrm{e}^{y_2}, ... $ (which are interpreted as the birth of new individuals)
 \begin{eqnarray} x \to (x \cdot \mathrm{e}^{y_0}, (x \cdot \mathrm{e}^{y_1}, x \cdot \mathrm{e}^{y_2}, \cdots )) \quad \mbox{ at a rate } \quad x^{-\alpha} \cdot \boldsymbol{\Lambda}( \mathrm{d} \mathbf{y}),  \label{eq:intensitysplit} \end{eqnarray}  where $\alpha >0$ is the self-similarity parameter. The projection $\Lambda_0$ of $\boldsymbol\Lambda$ on  its first coordinate is required to be a L\'evy measure, and its projection $\Bd\Lambda_1$ onto the second coordinate satisfies a mild integrability assumption (see \eqref{eq: generalised Levy measure}). The  coefficient $ \mathrm{a} \in \mathbb{R}$ encodes the drift term while  $\sigma^2$ controls the Brownian part of the evolution of the decoration along branches. See Section \ref{sec: construction of the ssMts} or \cite{BertoinJean2024SMta} for details. When the starting decoration $x>0$ is fixed, the law $ \mathbb{Q}_x$ is the distribution of the above genealogical tree started with an individual of decoration $x$.

 A crucial quantity to consider is the \textbf{cumulant function} defined by
\begin{equation} \label{def: kappa}
\begin{split}
    \kappa(\gamma)&= \frac{1}{2}\sigma^2\gamma^2 + \mathrm{a} \gamma + \int_{ \mathcal{S}} \boldsymbol{\Lambda}( \D \mathbf{u}) \left( \mathrm{e}^{\gamma  u_0}-1 - \gamma u_0 \mathbf{1}_{\{|u_0| \leq 1\}} + \sum_{i=1}^\infty \mathrm{e}^{\gamma u_i}\right), \\
	&= \psi(\gamma) + \int_{ \mathcal{S}} \boldsymbol{\Lambda}( \D \mathbf{u}) \bigg(\sum_{i=1}^\infty \mathrm{e}^{\gamma u_i}\bigg)
\end{split}
\end{equation}
where $\psi(\gamma)$ is the Laplace exponent of the L\'evy process with characteristics $(\sigma^2,  \mathrm{a}, \Lambda_0)$. This function $\kappa$ also appears as the Biggins transform of the underlying branching L\'evy process. In particular, if $\kappa$ takes strictly negative values then  the types of individuals is decaying over generations in expectation, and the genealogical tree of the above system of particles is compact and provides a ssMt. While if $\kappa$ only takes positive values then the tree explodes locally, see \cite{BertoinJean2016Leis}. The case when $\kappa$ touches $0$ while remaining non-negative was left aside in \cite{BertoinJean2024SMta} and is the context of this paper:

\begin{Assump}[Criticality I]\label{Assumption A} Suppose that for  $\gamma\geq 0$ we have $\kappa(\gamma)\geq 0$ and there exists $\omega_-\geq 0$ such that $\kappa(\omega_-) = 0$ and $\kappa$ is twice differentiable at $\omega_-$. Furthermore, There exists $\gamma_1 > \omega_-$ such that $\psi(\gamma_1)<0$ and $\kappa(\gamma_1)<\infty$.
\end{Assump}

We shall have a stronger technical \cref{Assumption B} analogue to the Cram\'er assumption in \cite{BertoinJean2024SMta} needed to study fine properties of ssMt. This assumption is omitted in the introduction for readability.

\begin{figure}
    \centering
    \includegraphics[width=6cm]{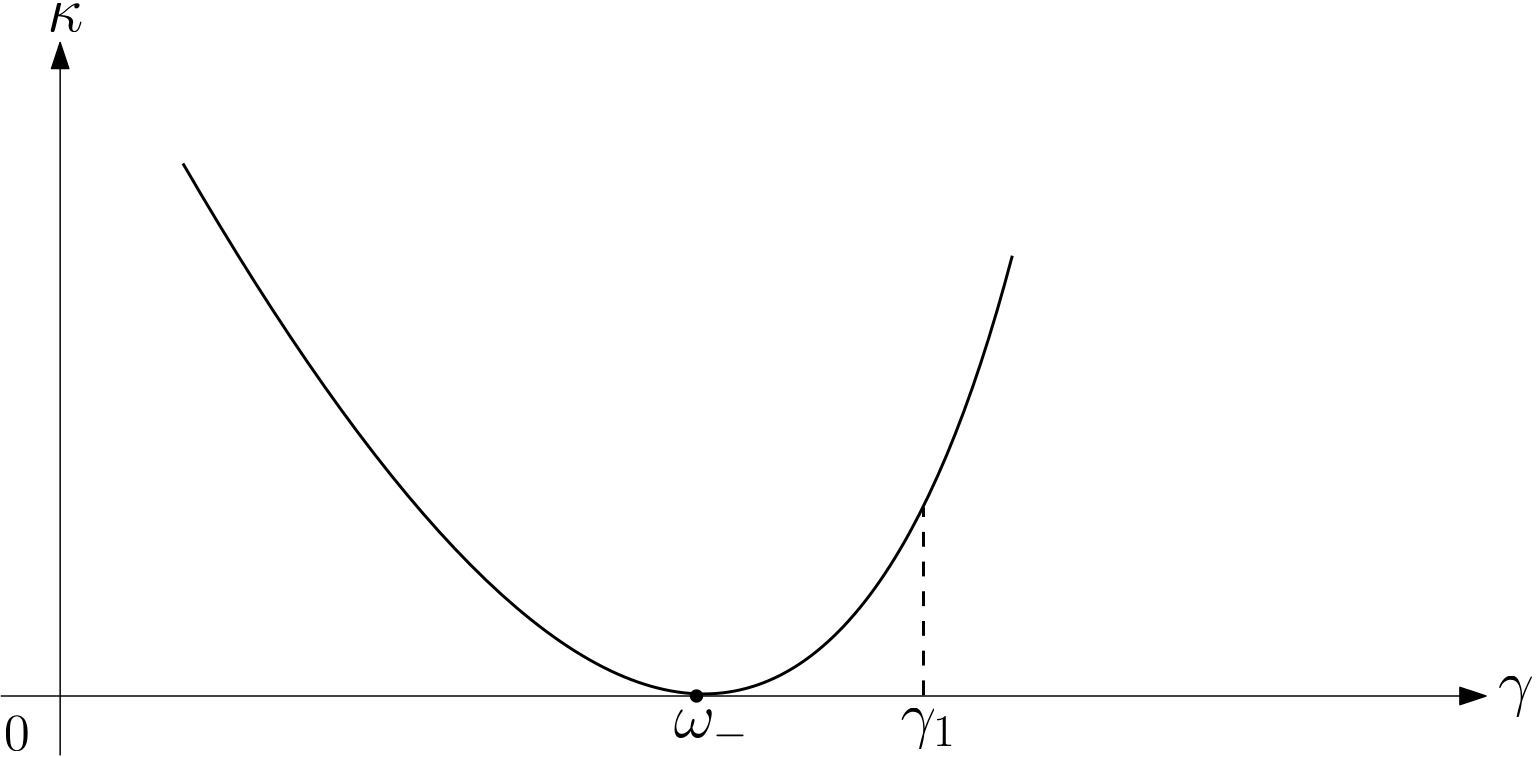}
    \caption{Illustration of the criticality on the cumulant function.
    }
    \label{fig:my_label}
\end{figure}

\begin{thm}[Construction of the critical ssMt]\label{thm: construct ssMt}
Under \cref{Assumption A}  the construction in \cite{BertoinJean2024SMta} can indeed be performed and it yields a family of laws $(\mathbb{Q}_x)$ of compact decorated random trees satisfying the Markov and self-similarity property. Furthermore, under \cref{Assumption B}, for any $x>0$ the Hausdorff dimension of the leaves $\partial T$ of a tree $T$ under $\mathbb{Q}_x$ is a.s.
$$ \mathrm{dim_{H}}(\partial T) = \frac{\omega_-}{\alpha}.$$
\end{thm}

Compared to \cite{BertoinJean2024SMta}, the construction of self-similar Markov trees in the critical case relies on the very recent result of A\"id\'ekon--Hu--Shi \cite{AidekonElie2024Bodt} which was in fact motivated by the above application. A few critical ssMt were already considered in \cite{BertoinJean2024SMta} but their existence was proved case-by-case using specific features of the models. This was notably the case for the ssMt of A\"id\'ekon and Da Silva \cite{aidekon2022growth} which arose from half-planar Brownian excursion. With the critical case at hand, the family of ssMt, which is now complete, is conjectured to describe all possible scaling limits of multi-type Bienaym\'e--Galton--Watson trees. We now move on to discuss the properties of critical self-similar Markov trees and their natural measures following the same strategy as for the subcritical case \cite{BertoinJean2024SMta}. 

\paragraph{Measures and spine decomposition.}
Any decorated tree $\mathtt{T}=(T,g)$ carries a natural Lebesgue measure $\lambda_{T}$ on its skeleton, and the length measures are obtained by using the decoration $g$ as density. Formally, for any $\gamma \geq 0$, we define the \textbf{$\gamma$-length measure} on $T$ as
$$\D \lambda^{\gamma}:=g^{\gamma-\alpha} \cdot \D \lambda_{T}.$$ In the subcritical case, those measures are defined as soon as $\gamma > \omega_-$ and Proposition 2.12 of \cite{BertoinJean2024SMta} even ensures that $\lambda^\gamma$ has finite expected mass when $\kappa(\gamma) <0$. In the critical case, we prove in \cref{sec: length measure} that $ \lambda^{\gamma}$ is still a finite measure when $\gamma>\omega_-$ a.s. but with \textit{infinite expected length}. Let us now turn to the equivalent of the harmonic measure of \cite[Chapter 2.3.3]{BertoinJean2024SMta}. Recall from Assumption \ref{Assumption A} that $\kappa(\omega_-)=0$. In this case, the  process
\[W_n = \sum_{|u| = n} (\chi(u))^{\omega_-},\] where $\chi(u)$ are the initial decorations of the individuals appearing in the genealogical tree (indexed by the Ulam's tree, see the construction of ssMt in \cref{sec: construction of the ssMts}) is a martingale. In the context of branching random walks, such a martingale is called an additive martingale. In the critical case, although positive, this martingale has a trivial limit  (see \cite[Theorem 3.3]{ShiZhan2016BRWE}), so the construction of the harmonic measure from \cite{BertoinJean2024SMta} needs to be adapted. The standard way to remedy this problem is to consider the \textbf{derivative martingale} defined by
\begin{equation}
D_n = -\sum_{|u| = n} (\chi(u))^{\omega_-}\log(\chi(u)).
\end{equation}
Standard results in the field of branching random walks (see \cite[Section 5]{ShiZhan2016BRWE}) ensures convergence of the derivative martingale towards a non-trivial positive limit ${D}_\infty$. This enables us to endow the ssMt $\mathtt{T}$ with a non-trivial measure $\mu$ of mass ${D}_\infty$ which plays the role of the harmonic measure in the subcritical case (we keep the same name in our context). Although not immediate from the definition, we will show that the harmonic measure $\mu$ is intrinsic, i.e. measurable with respect to the decorated tree only (as opposed to its genealogical representation). In fact $\mu$ can be obtained as a limit of the (intrinsic) length measures $\lambda^\gamma$ when $\gamma \downarrow \omega_-$
\[\lim_{\gamma \downarrow\omega_-} \frac{\kappa^{\prime\prime}(\omega_-)}{2}(\gamma-\omega_-)\lambda^{\gamma} = \mu, \]
at least along a subsequence as it was the case in the subcritical case \cite[Proposition 2.15]{BertoinJean2024SMta} (see \cref{thm: weak convergence length measure} for details). This convergence represents the most technical part of this work and requires delicate truncation estimates and fluctuation identities for L\'evy processes.

As in  \cite[Chapter 4 ]{BertoinJean2024SMta} those random measures are used to perform \textbf{spinal decomposition} of the underlying ssMt. The spinal decomposition originates in the setting of branching random walks and generalises to many different genealogical models, such as branching L\'evy processes and self-similar Markov trees. When doing so, we want to deal with the law of the decorated tree ${\tt T}$ together with a marked point $r$ sampled from the harmonic (or length) measure. The line segment between the root $\rho$ and the marked point $r$ is called the spine. The spinal decomposition theorem describes the law of the spine and the law of the subtrees dangling to the spine. In the critical case, the harmonic or length measures have infinite expected total mass, so we cannot directly bias the law of the tree by sampling a point according to those measures. However, this can be performed after a convenient cut-off on the decoration, see \cref{thm: spinal decomposition truncated}  in \cref{sec: spinal decomposition}. In particular, in contrast to \cite{BertoinJean2024SMta}, the L\'evy processes controlling the evolution of the decoration along the spine are now conditioned to stay below a barrier. The spinal decomposition is profound and has many applications.

The rest of the paper is organised as follows. In section 2, after a quick recap of the construction in \cite{BertoinJean2024SMta} we apply the results of \cite{AidekonElie2024Bodt} to construct critical self-similar Markov trees. We then import many results from the subcritical case using a "continuity argument" since critical ssMt can, after a slight perturbation, be transformed into a subcritical ssMt. This is used in particular to obtain a lower bound on their Hausdorff dimension, the spinal decomposition and characterisation of bifurcators. In section 3, we discuss the length and harmonic measures. A key idea imported from branching random walks is to consider the tree conditioned to stay below a barrier. In section 4, we present the spinal decomposition theorem with respect to the truncated harmonic measure.  Section 5, perhaps the most technical part of this paper, is devoted to analysing the relations between harmonic and lengths measures (\cref{thm: weak convergence length measure}).

\paragraph{Acknowledgments.} We thank Elie Aïdékon, Yueyun Hu, and Zhan Shi for stimulating discussions around \cite{AidekonElie2024Bodt} as well as Jean Bertoin and Armand Riera. The last two authors were supported by the China Scholarship Council. The first author is supported by "SuPerGRandMa", the ERC Consolidator Grant No 101087572.

\tableofcontents




\section{Background and construction of ssMt}
In this section we quickly recap the construction of decorated random trees by gluing decorated branches and prove the existence of critical self-similar Markov trees (\cref{thm: construct ssMt}). For more details, the reader is referred to  \cite[Chapters 1,2]{BertoinJean2024SMta}.

\subsection{Background} \label{sec: construction of the ssMts}

We first recall the definition of \textbf{characteristic quadruplet} $(\sigma^2, \mathrm{a}, \boldsymbol{\Lambda}; \alpha)$. Let $\mathcal{S} = [-\infty, \infty)\times \mathcal{S}_1$ where $\mathcal{S}_1$ is the set of non-increasing sequences $\Bd y = (y_1, y_2, \cdots)$ with $y_i\in [-\infty, \infty)$ and $\lim_{n\to\infty} y_n = -\infty$. We require that $\Bd\Lambda$ is a \textbf{generalised L\'evy measure} on the space $\mathcal{S}$, i.e., its projection $\Lambda_0$  to the first coordinate is a L\'evy measure and $\Bd\Lambda_1$ to the second coordinate satisfies
\begin{equation}\label{eq: generalised Levy measure} \Bd\Lambda_1(\{\mathrm{e}^{y_1}>\varepsilon\})<\infty, \quad \forall \varepsilon>0.
\end{equation}

The four entries have specific meanings in the construction. The tuple $(\sigma^2, \mathrm{a}, \Lambda_0)$ characterises a L\'evy process which is the input of the decoration. The measure $\Bd\Lambda_1$ is the intensity of the birth event. The parameter $\alpha>0$ is the \textbf{self-similarity index}. Given $(\sigma^2, \mathrm{a}, \boldsymbol{\Lambda}; \alpha)$, we first construct the law of decoration-reproduction process. Then the law of such process induces a \textit{family} of decorated branches indexed by the Ulam's tree. Finally, through a gluing procedure we build a decorated tree which is called the self-similar Markov tree with quadruplet $(\sigma^2, \mathrm{a}, \boldsymbol{\Lambda}; \alpha)$.

\bigskip

\noindent \textbf{Decoration-reproduction processes.} A \textbf{decoration-reproduction} process is a pair of processes $(f,\eta)$ where  $f$ is a non-negative right continuous with left limits (rcll) process on a line segment $[0, z]$ called the decoration process and where $\eta$ is a point process on $[0, z]$ called the reproduction process. Denote by $P_x$ the law of $(f, \eta)$ with $f(0) = x$.  We say that the family $(P_x)_{x\geq 0}$ is self-similar with index $\alpha>0$ if for each $x > 0$, the law  $P_x$ coincides with the rescaled pair $(f^{(x)},\eta^{(x)})= F^{\alpha}_x(f, \eta)$ under $P_1$ where
\begin{equation}\label{eq: scale f}
    f^{(x)}:[0, x^{\alpha}z]\to[0, \infty),\quad f^{(x)}(t) = xf(x^{-\alpha}t)
\end{equation}
\noindent and $\eta^{(x)}$ is the push-forward of $\eta$ under the map
\[
[0, z]\times(0, \infty)\to [0, x^{\alpha}z]\times(0,\infty), \quad (t, y)\mapsto (x^{\alpha}t, xy).
\]

We  construct such a family $(P_x)_{x>0}$ using characteristic quadruplet $(\sigma^2, \mathrm{a}, \boldsymbol{\Lambda}; \alpha)$ as follows: Let $\Bd{\mathbf{N}} = \Bd{\mathbf{N}}(\D t, \D y, \D \Bd y)$ be a Poisson random measure on $[0, \infty)\times \mathcal{S}$ with intensity $\D t\Bd\Lambda(\D y, \D \Bd y)$. Denote by $N_0(\D t, \D y)$ its projection to the first and the second coordinates and by  $\Bd{\mathbf{N}}_1 = \Bd{\mathbf{N}}_1(\D t, \D \Bd y)$ its projection to the first and the third coordinates. Let $B$ be a standard Brownian motion starting from $0$ independent of $\Bd{\mathbf{N}}$. Construct a process $(\xi(t))_{t\geq 0}$ by
\[\xi(t) = \sigma B(t) + \mathrm{a}t + \int_{[0, t]\times\mathbb{R}} y\mathds{1}_{\{|y|>1\}}N_0(\D t, \D y) + \int_{[0, t]\times \mathbb{R}} y\mathds{1}_{\{|y|\leq 1\}}N^c_0(\D t, \D y). \]
\noindent The compensate point process is defined by $N^{c}_0(\D t, \D y) : = N_0(\D t, \D y) - \D t \Lambda_0(\D y)$. The process $(\xi(t))_{t\geq 0}$ is a L\'evy process with life-time $\zeta := \inf\{t>0:\xi(t) = -\infty\}$. We may consider the L\'evy process $\xi$ started from $b\in \mathbb{R}$ by adding a constant $b$ to the display above. Define its Laplace exponent $\psi(\gamma)$ by the following equation \footnote{Such expectation could be infinite in general. The condition that $\psi(\gamma)<\infty$ is included in the subcritical and critical condition in the next section. }
\begin{equation}\label{eq: Laplace exponent}
E[\exp(\gamma \xi(t))] = E[\exp(\gamma \xi(t))\cdot\mathds{1}_{\{t<\zeta\}}]=\exp(t\psi(\gamma)), \quad \gamma>0.
\end{equation}

\noindent By the L\'evy Khintchine formula, we have ($\mathbb R^*:=\mathbb R \cup \{-\infty\}$)
\begin{align}\label{eq: def psi}
    \psi(\gamma) = \frac{1}{2}\sigma^2\gamma^2 + \mathrm{a}\gamma + \int_{\mathbb{R}^*}(\mathrm{e}^{\gamma y} - 1 - \gamma y\mathds{1}_{|y|\leq 1})\Lambda_0(\D y).
\end{align}

We next apply the Lamperti transformation to the process $(\xi(t))_{t\ge 0}$. Consider the time change
\begin{align*}
     \epsilon(t):=\int_0^t \exp(\alpha\xi(s))\D s \quad\text{ for }0\le t\le \zeta, \quad z=\epsilon(\zeta-)=\int_0^{\zeta} \exp(\alpha\xi(s))\D s.
\end{align*}
The (random) function $\epsilon(t):[0,\zeta)\to [0,z)$ is a bijection a.s.. Write $\tau$ as the reciprocal bijection so that a.s. $\int_{0}^{\tau(t)}\exp(\alpha \xi(s))\D s=t$. By the Lamperti transformation the process $(X(t):=\exp (\xi(\tau(t))))_{t\in[0, z)}$ is a positive self-similar Markov process (pssMp) starting from $1$ with scaling exponent $\alpha$. Specifically, for each $x>0$, the scaled process $(xX(x^{-\alpha}t))_{t\in [0, x^\alpha z)}$ has the same law as $X$ starting from $x$. The point $0$ serves as a cemetery point for $X$. We always set $X(z):=0$ when $z<\infty$ such that $X$ is a rcll process.

We define the reproduction process $\eta$ using the other projection $\mathbf{N}_1(\D t, \D\Bd y)$. We rewrite each atom $(s,\boldsymbol{y})$ of $\bf N_1$ (possibly repeated according to their multiplicities), as pairs $(s, y_{\ell})_{\ell\ge 1}$ when $\boldsymbol{y}\neq (-\infty,-\infty,\dots)$. Recall the exponential functional $\epsilon(t)$. Set
\begin{equation}\label{eq: def eta}
    \eta:=\sum \mathds{1}_{\{ \epsilon(s)\le z\}} \delta_{(\epsilon(s),\exp (\xi(s-)+y_l))},
\end{equation}
where $\delta_{(t,x)}$ is the Dirac measure at $(t,x)$ and the sum is taken over all the pairs $(s, y_{\ell})_{\ell\ge 1}$. Denote the law of $(X, \eta)$ by $P_x$ if $X$ starts from $x$. We see that the family $(P_x)_{x>0}$ is self-similar with parameter $\alpha$. For brevity, we write $ P$ for $ P_1$.

\bigskip

\noindent \textbf{Decorated-reproduction family.} Let $\mathbb{U} = \bigcup_{n\geq 0} (\mathbb{N}^*)^n$ be the Ulam tree with the convention $(\mathbb{N}^*)^0 = \{\varnothing\}$. Set $\mathbb{U}^* = \mathbb{U}\backslash \{\varnothing\}$. The leaves of the Ulam tree are denoted by $\partial\mathbb{U} := \mathbb{N}^{\mathbb{N}}$. We write $|u| = n$ (resp. $|u| = \infty$) if $u\in(\mathbb{N}^*)^n$ (resp. $u\in\partial\mathbb{U}$) for the generation of $u$, $u-$ for the parent of $u$, $uv$ for concatenation of $u$ and $v$, and $u_k$ for the ancestor of $u$ at generation $k$.

With the family $(P_x)_{x>0}$, we construct the decorated-reproduction family, which is a random family of decoration-reproduction processes $(f_u,\eta_u)_{u\in \mathbb{U}}$. It is a particle system where each $u$ is an individual (possibly fictitious) and $(f_u, \eta_u)$ describes the trait and birth event along its life. Each particle $u$ is assigned a type $\chi(u)$. By convention, we set $\chi(u) = 0$ if an individual is fictitious. At generation $0$, there is an individual with type $\chi(\varnothing) = x>0$. We then sample $(f_{\varnothing}, \eta_{\varnothing})$ under $P_x$. We enumerate the atoms of $\eta_\varnothing$ by $(t_1,y_1),(t_2,y_2),\dots$, and complete the sequence with fictitious individuals to get an infinite sequence. The individuals in the first generation have types $\chi(i) = y_i$. Inductively, for each $u\in \mathbb{U}$ with type $\chi(u)$, sample $(f_u,\eta_u)$ under the law $P_{\chi(u)}$ independently. We then set $\chi(ui)=y_{ui}$ for $i\in \mathbb{N}$ as the types of children of $u$, where $ ( t_{u1}, y_{u1}),(t_{u2}, y_{u2}),\dots$ are the atoms of $\eta_{\chi(u)}$. We repeat this procedure to obtain the next generation by using independent decoration-reproduction processes for different individuals. We denote by $\mathbb{P}_x$ the law of the family of decoration-reproduction processes $(f_u,\eta_u)_{u\in \mathbb{U}}$ when the ancestor $\varnothing$ has the type $x$ and $\mathbb{E}_x$ the corresponding expectation. For brevity, we write $\mathbb P$ for $\mathbb P_1$ and $\mathbb E$ for $\mathbb E_1$.

\bigskip

\noindent \textbf{Construct trees by gluing branches.}
Let $(f_u, \eta_u)_{u\in\mathbb{U}}$ be a family of decoration-reproduction processes. For each individual $u\in \mathbb{U}$, we view the decoration process $(f_u, \eta_u)$ as a decorated branch $([0, z_u], d, 0, f_u)$ with marks $t_{ui}$ (recall that $(t_{ui}, y_{ui})$ are atoms of $\eta_u$ and $z_u$ the lifetime of $f_u$). The distance $d$ is the Euclidean distance. Let ${\tt T}_0 = ([0, z_\varnothing], d, 0, f_\varnothing)$. Inductively, we let ${\tt T}_{n+1}$ be the decorated trees obtained from gluing the branches $([0, z_u], d, 0, f_u)$ with $|u| = n+1$ onto ${\tt T}_n$ at the corresponding $t_{u}$ lying in $T_n$. The gluing operation leads to a limiting metric space $(T, d_T, \rho)$ with a decoration function $g$ as $n\to\infty$. For each point $x$ in $T$, we set $g(x) = \sup f_u(x_u)$ where the supremum is taken over $x_u$ in the branch indexed by $u\in \mathbb U$ identified to $x$. The following lemma provides a sufficient condition for the space $(T, d_T, \rho,g)$ to be a compact decorated real tree (see also \cite[Section 2]{senizergues2020growing}). Define the norm $||f_u||:=z_u+\sup_{0\le t\le z_u}f_u(t)$. We say that $(x_i, i\in I)$ is a \textbf{null family} if for $\varepsilon>0$, there are only finitely many $x_i>\varepsilon$.

\begin{lem}[{\cite[Lemma 1.5]{BertoinJean2024SMta}} Compactly glueable]\label{thm: gluing building blocks}
Suppose that the family $(f_u, \eta_u)$ satisfies
\begin{equation}\label{eq: condition 1}
(||f_u||, u\in\mathbb{U}) \quad \mbox{is a null family}
 \qquad \mbox{and}  \qquad
\lim_{k\to\infty} \sup_{\bar{u}\in\partial\mathbb{U}} \sum^{\infty}_{n=k} z_{\bar{u}(k)} = 0.
\end{equation}
Then $(T, d_T, \rho, g)$ is a compact decorated real tree.
\end{lem}

\noindent \textbf{Topology of the space of (measured) decorated trees}.
\label{sec: topology} We close this section with a brief note on the topology on the space of decorated compact real tree, see \cite[Section 1.4]{BertoinJean2024SMta}. We say $\mathtt{T}:=(T,d_T,\rho,g) $ and $\mathtt{T}^\prime:=(T^\prime,d_T^\prime,\rho^\prime,g^\prime)$ are isomorphic if there exists a bijective isometry $\phi: (T,d_T)\to (T^\prime,d_T^\prime)$ such that $\phi(\rho)=\rho^\prime$ and  $g^\prime=g\circ \phi^{-1}$. Denote by $\mathbb T$ (resp. $\mathbb{T}^\bullet$ or $\mathbb{T}_m$) the space of equivalence (up to isomorphisms) classes of  decorated compact real trees (resp. with an additional point, or a finite Borel measure).  We always abuse $\mathtt{T}=(T,d_T,\rho,g)$ or $\mathtt{T}^\bullet=(T,d_T,\rho,g,r)$ to represent an equivalence class in the sequel (where $r$ is an additional point), since the quantities and properties we will consider are invariant under isomorphisms.  These spaces are endowed with $ d_{\mathbb{T}}, d_{\mathbb{T}^\bullet}, d_{\mathbb{T}_m}$, which are adaptations of the Gromov-Hausdorff-pointed/Prokhorov distances that also take the decorations into account as hypographs. They are all Polish spaces. See \cite[Section 1.4]{BertoinJean2024SMta} for details.

\subsection{Critical self-similar Markov trees}

For a generic choice of quadruplet $(\sigma^2, \mathrm{a}, \Bd\Lambda; \alpha)$, the decorated-reproduction family $(f_u,\eta_u)_{u\in \mathbb{U}}$ might not be compactly glueable. We will see that the cumulant function $\kappa(\gamma)$ defined in \eqref{def: kappa} plays an important role in that respect. We also frequently use the moment generating function of the types $\mathcal{M}$, which is defined by
\begin{equation}\label{eq: Merlin transform}
\mathcal{M}(\gamma) := \mathbb{E}_1\bigg[\sum^{\infty}_{i=1}(\chi(i))^{\gamma}\bigg] = 1 -\frac{\kappa(\gamma)}{\psi(\gamma)}, \quad \mbox{for } \psi(\gamma)<\infty.
\end{equation}
The last equation is from \cite[Lemma 3.8]{BertoinJean2024SMta}. The function $\kappa$ and $\mathcal{M}$ are both convex by definition. In \cite{BertoinJean2024SMta} the authors work under the subcritical condition that $\inf_{\gamma}\kappa(\gamma)<0$. Under this assumption, $(f_u, \eta_u)_{u\in\mathbb{U}}$ is $\mathbb{P}_x$-a.s. compactly glueable. The resulting decorated tree is called in these pages a subcritical self-similar Markov tree associated with quadruplet $(\sigma^2, \mathrm{a}, \Bd\Lambda; \alpha)$. In this paper, we prove that when $\inf_{\gamma} \kappa(\gamma)=0$, the \textbf{critical case}, under \cref{Assumption A} or the more exotic \cref{Assumption A'} below, the construction can still be performed:

\begin{Assump}[Criticality II]\label{Assumption A'}
Suppose that for $\gamma\geq 0$ we have $\kappa(\gamma)\geq 0$ and there exists $\omega_-\geq 0$ such that $\kappa(\omega_-) = 0$ and $\kappa'(\omega_-)<0$. Furthermore, There exists $\gamma_1 > \omega_-$ such that $\psi(\gamma_1)<0$.
\end{Assump}

\begin{prop}[Existence of critical ssMt]\label{prop: construct ssMt}
With the notation above, under \cref{Assumption A} or \cref{Assumption A'}, the decorated branches $(f_u, \eta_u)_{u\in\mathbb{U}}$ under $\mathbb{P}_x$ are a.s. compactly glueable in the sense of \cite[Definition 1.4]{BertoinJean2024SMta} and provides a decorated real tree whose law is denoted by $\mathbb{Q}_x$.
\end{prop}

\begin{rem} We prove below \cref{prop: construct ssMt} using the main result of \cite{AidekonElie2024Bodt}. We shall however need to go over the proof of \cite{AidekonElie2024Bodt} later to gather estimates needed for the upper bound of Hausdorff dimension in \cref{lem: height truncated tree}. The decorated trees under law $\mathbb{Q}_x$ may still be called ssMt (with characteristics $(\sigma^2, \mathrm{a}, \boldsymbol\Lambda ; \alpha))$: the self-similarity is inherited from that of the decorated-reproduction family, and the Markov properties stated in \cite[Chapter 4]{BertoinJean2024SMta} generalise to the critical self-similar Markov trees by applying the same arguments.
\end{rem}
\begin{proof}
The proof of \cref{thm: construct ssMt} reduces to checking the two conditions \eqref{eq: condition 1} of \cref{thm: gluing building blocks}. The second condition is a direct consequence of \cite[Theorem 1.3]{AidekonElie2024Bodt} given that the following limit exists and satisfies
\[\lim_{x>0} \frac{-\log\mathbb{P}(z_{\varnothing}>x)}{\log x} \geq \frac{\gamma_1}{\alpha} > \frac{\omega_-}{\alpha}. \]

\noindent The existence of $\gamma_1>0$ with $\psi(\gamma_1) < 0$ ensures that $\mathbb{E}[z_{\varnothing}^{\gamma_1/\alpha}] < \infty$ by \cite[Lemma 9.1]{BertoinJean2024SMta}, which implies $\mathbb{P}(z_{\varnothing}>x)\leq Cx^{-\gamma_1/\alpha}$. Without assuming the limit exists, we could couple $(z_{u}, u\in\mathbb{U})$ with i.i.d random variable $(Y_u, u\in\mathbb{U})$ such that $z_u \leq (\chi(u))^{\alpha}\cdot Y_u$ almost surely and $\mathbb{P}(Y>x) = Cx^{-\gamma_1/\alpha}$. Then by \cite[Theorem 1.3]{AidekonElie2024Bodt}, the second condition of \eqref{eq: condition 1} holds with $z_{u(n)}$ replaced by $(\chi(u))^{\alpha}\cdot Y_u$.

The first condition in  \eqref{eq: condition 1} results from the decaying of the types over generations. We simplify $\sup_{0\le t\le z_u} f_u(t)$ as $\sup f_u$. By successively applying self-similarity, \eqref{eq: Merlin transform}, and \cite[Lemma 9.1]{BertoinJean2024SMta},
\begin{align*}
     \mathbb{E}_1\left[ \sum_{|u|\le n}\sup f_{u}^{\omega_-}\right]= \mathbb{E}_1\left[ \sum_{|u|\le n}\chi(u)^{\omega_-} \right]\mathbb{E}_1[\sup f_{\varnothing}^{\omega_-}]=(n+1)\mathbb{E}_1[\sup f_{\varnothing}^{\omega_-}]<\infty.
\end{align*}
It follows that there are only finite many $|u|\le n$ with $\sup f_u > \varepsilon $ for each $n\in \mathbb{N}$. For this $\varepsilon > 0$, there exist $\delta > 0$ and $C>0$, such that $\mathbb{P}_x(\chi(1)>\delta) > C$ for $x>\varepsilon$. By Markov property of the decoration-reproduction process, for each $u\in \mathbb U$ with $\sup f_u > \varepsilon$, with probability greater than $C$, one offspring of $u$ has type greater than $\delta$. Therefore, we have
\begin{equation}\label{eq: type and decoration}
      \mathbb{P}\left(\sup_{|u|\ge n+1}\chi(u)>\delta\right)\ge C \mathbb{P}\left(\sup_{|u|\ge n}\sup f_u>\varepsilon\right).
\end{equation}
With \cref{Assumption A} or \cref{Assumption A'}, we have $\lim_{n\to\infty}\sup_{|u| = n} \chi(u) = 0$ a.s. from  \cite[Lemma 3.1]{ShiZhan2016BRWE}. Together with \eqref{eq: type and decoration}, it implies that $\lim_{n\to \infty}\mathbb{P}(\sup_{|u|\ge n} \sup f_u>\varepsilon)=0$. We conclude that $(\sup f_u,u\in \mathbb{U})$ is a null family. In the same way, we get $(z_u,u\in \mathbb{U})$ is a null family. Then \eqref{eq: condition 1} follows from the definition of $||f_u||$.
\end{proof}

\subsection{Approximation by subcritical ssMt}\label{sec: approximation}
In this section, we introduce several approaches to couple a critical ssMt with an approximating sequence of subcritical ssMt. These couplings will be used to lift properties from subcritical ssMt to critical ssMt (lower bound on Hausdorff dimension, spinal decomposition...).

A simple way to couple self-similar Markov trees is to couple their  decoration-reproduction processes for each decorated branch. Fix a (critical) characteristic quadruplet $(\sigma^2, \mathrm{a}, \boldsymbol\Lambda ; \alpha)$ and fix  $\varepsilon > 0$. Recall from \cref{sec: construction of the ssMts} the construction of the decoration–reproduction process $(X,\eta)$ from a Lévy process $\xi$ and a point process $\mathbf{N}_1$ whose atoms are of the form $(s,\mathbf y)$. We provide three different processes $(X^{i, \varepsilon}, \eta^{i, \varepsilon})$ ($i = 1,2,3$) each with law $P^{i, \varepsilon}$ coupled with $(X, \eta)$ with law $P$.

\begin{description}
    \item[Adding negative drift] $(X^{1, \varepsilon}, \eta^{1, \varepsilon})$ is obtained by adding a drift $-\varepsilon$ to the underlying L\'evy process $\xi$. Specifically, set $\xi^{1, \varepsilon}_t = \xi_t - \varepsilon t$ and denote by $X^{1, \varepsilon}$ the Lamperti transform of $\xi^{1, \varepsilon}$. We define $\eta^{1,\varepsilon}$ by \eqref{eq: def eta}, with $\mathbf{N}_1$ unchanged and $\xi$ replaced by $\xi^{1,\varepsilon}$.
    \item[Adding killing] $(X^{2, \varepsilon}, \eta^{2, \varepsilon})$ is constructed similarly by adding killing with rate $\varepsilon$ to $\xi$. We denote by $\mathrm{e}(\varepsilon)$ an independent exponential time with parameter $\varepsilon$. The life time $z^{2, \varepsilon}$ of $(X^{2, \varepsilon}, \eta^{2, \varepsilon})$ equals $\epsilon(\zeta \wedge \mathrm{e}(\varepsilon) )$ where $\epsilon$ is the Lamperti time transform of $\xi$. We set $(X^{2, \varepsilon}, \eta^{2, \varepsilon}) = (X\cdot \mathds{1}_{[0, \epsilon(z^{2, \varepsilon})]}, \eta\cdot \mathds{1}_{[0, \epsilon(z^{2, \varepsilon})]\times\mathbb{R}^*})$. We always
   couple the exponential times $\mathrm{e}(\varepsilon)$ for different $\varepsilon$ such that they are decreasing in $\varepsilon$ (i.e.~$\mathrm{e}(\varepsilon)\uparrow \infty$ as $\varepsilon\downarrow 0$) a.s..

    \item[Lowering the reproduction] $(X^{3, \varepsilon}, \eta^{3, \varepsilon})$ is obtained by shifting downwards the atoms $(s, \Bd y)$ of $\mathbf{N}_1$ by $\varepsilon$. We set $X^{3, \varepsilon} = X$ and
    \[\eta^{3, \varepsilon} = \sum\mathds{1}_{\{\epsilon(s)\leq z\}}\delta_{(\epsilon(s), \exp(\xi(s-) + y_{\ell} - \varepsilon))}. \]
\end{description}

Denote by $P^{i, \varepsilon}_x$ the law of $(X^{i, \varepsilon}, \eta^{i, \varepsilon})$ starting from $x$ for $i=1,2,3$. See Figure \ref{fig:decoration} for illustration of the effects of the case 1,2 and 3 on a given decoration-process.

\begin{figure}
    \centering
    \includegraphics[width=15cm]{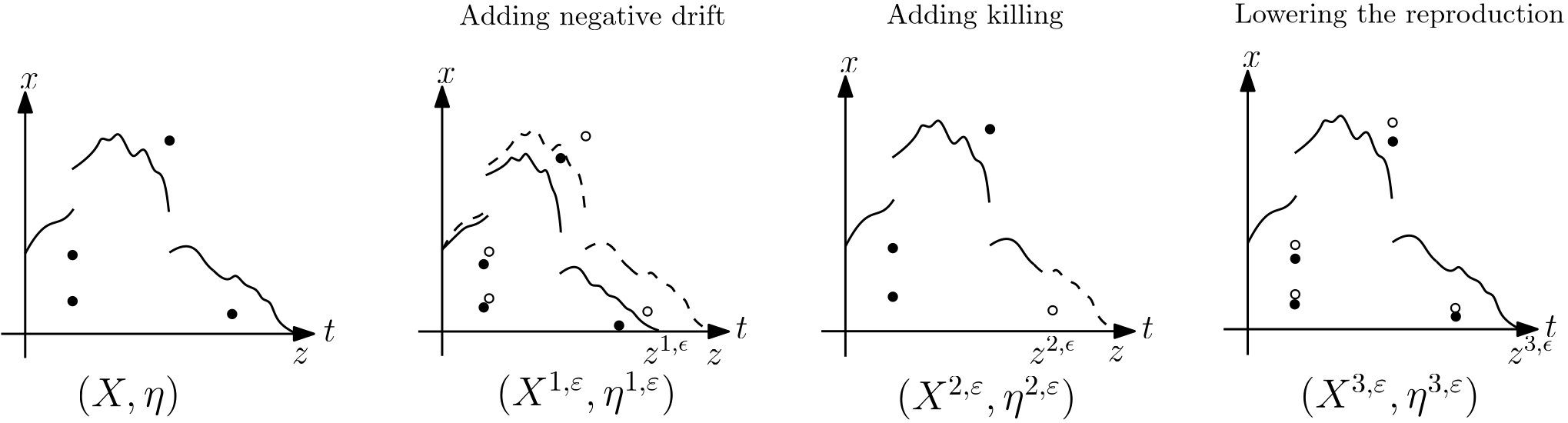}
    \caption{Illustration of a decoration-reproduction process $(X,\eta)$ and $(X^{i,\varepsilon},\eta^{i,\varepsilon})$ ($i=1,2,3$) coupled with it. The dots represent the locations of the atoms of the reproduction. The dashed lines and circles represent the original decoration–reproduction process $(X,\eta)$. In the drift case (1), the atoms of the reproduction are both shifted in time and space. In the killing case (2), the atoms are the same, the decoration-reproduction process is just possibly killed earlier. In the reproduction case (3), the atoms stay at the same position but are multiplied by $ \mathrm{e}^{- \varepsilon}$.}
    \label{fig:decoration}
\end{figure}

In each of the above three cases, in the obvious coupling of the atoms $(t_j^{i, \varepsilon}, x_j^{i, \varepsilon})_{j \geq 1}$ of $\eta^{i, \varepsilon}$ with the atoms $(t_j, x_j)_{j\geq 1}$ of $\eta$ we have $x_j^{i, \varepsilon} \leq x_j$, and the life times satisfy $z^{i, \varepsilon} \leq z$. These properties enable us to iterate the above coupling and construct two decorated-reproduction families $(f_u, \eta_u)_{u\in\mathbb{U}}$ and $(f_u^{i, \varepsilon}, \eta_u^{i, \varepsilon})_{u\in\mathbb{U}}$ of laws $\mathbb{P}$ and $\mathbb{P}^{i, \varepsilon}$. Since $(f_u, \eta_u)_{u\in\mathbb{U}}$ is compactly glueable, so is $(f_u^{i, \varepsilon}, \eta_u^{i, \varepsilon})_{u\in\mathbb{U}}$. Their gluings yield two decorated compact trees $\mathtt{T}$ and $\mathtt{T}^{i, \varepsilon}$ together with a natural projection
$ p^{i, \varepsilon} : \mathtt{T} \to \mathtt{T}^{i, \varepsilon}.$

\begin{prop}[Subcritical approximations]\label{prop: approximate critical from subcritical} Suppose that  $(\sigma^2, \mathrm{a}, \boldsymbol\Lambda ; \alpha)$ is a critical characteristic quadruplet satisfying \cref{Assumption A} or \cref{Assumption A'}. For any $ \varepsilon >0$, the decorated trees $\mathtt{T}^{i, \varepsilon}$ ($i \in \{1,2,3\}$) are subcritical ssMt with characteristics, respectively, given by  $(\sigma^2, \mathrm{a}-\varepsilon, \Bd\Lambda;\alpha)$, $(\sigma^2, \mathrm{a}, \Bd\Lambda + \varepsilon\delta_{(-\infty, (-\infty,\ldots))};\alpha)$ and $(\sigma^2, \mathrm{a}, \Bd\Lambda^{\varepsilon};\alpha)$ with
    \[\Bd\Lambda^{\varepsilon}(\D y, \D\Bd y) = \Bd\Lambda(\D y, (\D(y_i + \varepsilon))_{i\geq 1}).\]
In all three cases we have  ${\tt T}^{i, \varepsilon} \to \tt T$ a.s. for the Gromov--Hausdorff--decorated distance and furthermore $p^{i, \varepsilon} : {\tt T}^{i, \varepsilon} \mapsto \tt T$ is $1$-Lipschitz.
\end{prop}

\begin{proof} The claim that $\mathtt{T}^{i, \varepsilon}$ are subcritical ssMt with the above quadruplets is clear from the construction. Notice that in the second case (adding killing), the subtree $\mathtt{T}^{3, \varepsilon}$ is obtained by pruning $\mathtt{T}$ at countably many points at most.  In the third case, lowering the reproduction amounts to scale the branches $\mathtt{T}$ by $\mathrm{e}^{- \varepsilon n}$ where $n$ is the generation of the individual. In those two cases, clearly the projection is $1$-Lipschitz. In case $1$, the Lamperti time transform over any interval $(s, t)$ satisfies $$\epsilon^{1,\varepsilon}(t) - \epsilon^{1,\varepsilon}(s)<\epsilon(t)- \epsilon(s).$$ Iterating the argument, we get that the canonical projection is also $1-$Lipschitz.  We finally prove the convergence of the approximating subcritical trees. We first see that the first branch converges as $\varepsilon\to 0$ since the process converges a.s.. We make it more precisely in the second case. If $f_{\varnothing}(z_{\varnothing}-)>0$, we have $(f^{2, \varepsilon}_{\varnothing}, \eta^{2, \varepsilon}_{\varnothing}) = (f_{\varnothing}, \eta_{\varnothing})$ when $\varepsilon$ is small enough. Otherwise, when $f_{\varnothing}(z_{\varnothing}-)=0$, as $\varepsilon\downarrow 0$, the Skorohod distance between $f_{\varnothing}$ and $f_{\varnothing}^{2,\varepsilon}$ goes to $0$ a.s.. A moment's thought shows that for any $N \geq 1$, as $ \varepsilon \to 0$, the subtree $\mathtt{T}^{i, \varepsilon}(N) \subset \mathtt{T}^{i, \varepsilon}$ spanned by the individuals $u \in \{1,... , N\}^{N}$ converges in the Gromov--Hausdorff--decorated sense $d_{\mathbb{T}}$ (no measure) towards the analogous subtree in $\mathtt{T}(N) \subset\mathtt{T}$. However, by the monotonicity properties of the coupling we have $ d_{\mathbb{T}}(\mathtt{T}^{i, \varepsilon}, \mathtt{T}^{i, \varepsilon}(N)) \leq d_{\mathbb{T}}(\mathtt{T}, \mathtt{T}(N))$. The latter tends to $0$ by the compactly glueability of $(f_u, \eta_u)_{u \in \mathbb{U}}$. This enables us to exchange the limits $N \to \infty$ and $ \varepsilon \to 0$ and get the statement.
\end{proof}

\begin{rem}[Intrinsic approximation] \label{rem: intrinsic approx} Contrary to the third case which requires the underlying family of decoration-reproduction, adding a killing or a drift is an operation which is geometrically intrinsic, that is, (the isometry class of) the random decorated tree $\mathtt{T}^{i, \varepsilon}$ is a measurable function of  $\mathtt{T} \in \mathbb{T}$ for $i=1$ and $i=2$ (together with additional independent randomness). Let us give some explanation and leave the technical details to the reader: For $x\in T$, we view the function $g:\llbracket \rho, x\rrbracket\to \mathbb{R}_+$ as a function $g:[0, d(\rho, x)] \to \mathbb{R}_+$. We can then take the inverse Lamperti transform to get the underlying L\'evy process, add a drift $-\varepsilon$ or a killing rate $ \varepsilon$, and finally take the Lamperti transform back to obtain the corresponding branch in $\mathtt{T}^{i, \varepsilon}$. For different $x, y\in {\tt T}$, such chain of transformations agrees on $\llbracket \rho, x\wedge y \rrbracket$ where $x\wedge y$ is the last common point of the geodesics. The point is that, by the memoryless property of the exponential variables, these constructions can be performed for all $x \in \mathtt{T}$ simultaneously and coherently, which ensures the well-posedness of the construction.
\end{rem}

\subsection{Examples}

We provide a few examples of critical ssMt. First of all, in the light of Proposition \ref{prop: approximate critical from subcritical}, many subcritical ssMt can conversely be transformed into critical ones by possibly lowering the killing rate, adding a positive drift or increasing the reproduction. The reader may try this procedure on the examples of \cite[Chapter 3]{BertoinJean2024SMta}. Perhaps the most important example of critical ssMt so far is the one of A\"id\'ekon and Da Silva in \cite{aidekon2022growth} (see  \cite[Section 3.4]{BertoinJean2024SMta}) which has attracted considerable recent attention:

\begin{Ex}[A\"id\'ekon \& Da Silva \cite{aidekon2022growth}]\label{ex: 1} Let $\mathtt{T}$ be the ssMt under $\mathbb{P}$ with characteristic quadruplet $(0,\mathrm{a}_{\rm ads},\Bd\Lambda_{\rm ads};1)$ where $\Bd\Lambda_{\rm ads} $ is defined by
\begin{align*}
    &\int F(\mathrm{e}^{y_0},(\mathrm{e}^{y_1},\dots))\Bd\Lambda_{\rm ads}(\D y_0,\D (y_i)_{i\ge 1}) \\
    =& \frac{2}{\pi} \left(\int_{1/2}^1\frac{\D x}{(x(1-x))^2} F(x,(1-x,0,\dots))+\int_{0}^\infty\frac{\D x}{(x(1+x))^2} F(x+1,(x,0,\dots))\right),
\end{align*}
and $\mathrm{a}_{\rm ads}=-\frac{4}{\pi}+\frac{2}{\pi} \int_{-\log 2}^{\infty} \D y (y\mathds{1}_{|y|\le 1}-(\mathrm{e}^y-1))\frac{\mathrm{e}^{-y}}{(\mathrm{e}^y-1)^2}$ (see \cite[Example.3.13]{BertoinJean2024SMta}). An explicit calculation shows that  $\kappa(\gamma)=2(\gamma-2)\tan(\frac{\pi \gamma}{2})$ when $3/2<\gamma<5/2$. We check that \cref{Assumption A} (with $\omega_- = 2$) and \cref{Assumption B} are satisfied. Thus $\mathtt{T}$ is indeed a critical ssMt. In fact, this ssMt can be seen as a variant of Brownian CRT dressed with an uncommon decoration:
A\"id\'ekon \& Da Silva \cite{aidekon2022growth} proved that it appears in the tree structure underneath a half-planar Brownian excursion where the vertical displacement encodes the tree structure and the horizontal displacement the decoration. This tree also appears concerning critical $O(n)$-loop model on planar maps \cite[Example 3.14]{BertoinJean2024SMta} as well as in conjectured scaling limits of random flat disks see \cite{budd2025discrete}.
\end{Ex}

Another critical example with finite intensity is taken from  \cite[Example 3.3]{BertoinJean2024SMta}:

\begin{Ex}[Branching Bessel processes] Consider the characteristic quadruplet $(1,-\sqrt{2},\delta_{(-\infty,(0,0,-\infty))};2)$. The underlying continuous time branching process is called the binary branching Brownian motion with drift $-\sqrt{2}$. By a direct calculation, its cumulant function is given by
    \begin{align*}
        \kappa(\gamma)=\frac{\gamma^2}{2}-\sqrt{2}\gamma+1.
    \end{align*}
\noindent The characteristic quadruplet defines a critical ssMt since Assumption \ref{Assumption A} holds with $\omega_- = \sqrt2$.
\end{Ex}




\begin{center} \hrulefill \ In what follows we suppose \cref{Assumption A} and the forthcoming \cref{Assumption B}. \hrulefill  \end{center}

\section{Measures on critical ssMt}\label{sec: length measure}

On subcritical self-similar Markov trees, two different kinds of measures were constructed in \cite{BertoinJean2024SMta}, namely length measures supported on the skeleton of tree and the harmonic measure supported on its leaves. In this section, we construct their analogous versions in the critical case: Given a decorated tree $\mathtt{T}$, for $\gamma >0$ we consider the $\gamma$-length measure $\lambda^\gamma$ whose density with respect to the natural Lebesgue measure $\mathrm{d}\lambda_T$ on its skeleton is given by $g^{\gamma-\alpha}$. In our critical ssMt case, following the argument in \cite[Section 2.3.1]{BertoinJean2024SMta}, we get
\[
\mathbb{Q}_1( \lambda^\gamma(T)) =
\mathbb{E}\bigg[\sum_{u\in\mathbb{U}}\int^{z_u}_0 f_u(t)^{\gamma-\alpha}\D t\bigg] = \mathbb{E}\bigg[\sum_{u\in\mathbb{U}}(\chi(u))^{\gamma}\bigg]E_1\bigg[\int^{z}_0 f(t)^{\gamma_\alpha}\D t \bigg].
\]
Since $\mathbb{E}[\sum_{|u|=n}(\chi(u))^{\gamma}] = \mathcal{M}(\gamma)^n$ by the branching property, the total mass  $\lambda^\gamma(T)$  has infinite expectation and it is not clear whether the measure $\lambda^{\gamma}$ is finite. We show in \cref{prop: expected length} that under \cref{Assumption A} the answer is yes for $\gamma>\omega_-$. In fact,  we shall prove in \cref{sec: tree below barrier} that the measure restricted to the tree below a barrier has finite expectation for $\omega_-<\gamma< \gamma_1$.

To study such measures we first need to introduce some tools from  branching random walks theory and fluctuation theory for L\'evy processes.

\subsection{Tools from Branching processes}\label{sec: tools from branching processes}

\noindent \textbf{Branching L\'evy processes.} Branching L\'evy processes were formally introduced by Bertoin \& Mallein in \cite{BertoinJean2019IRPM}. It corresponds to the branching structure of the underlying decoration reproduction family $(f_u, \eta_u)_{u \in \mathbb{U}}$ before the Lamperti transformation. We refer to \cite{MalleinBastien2023Anas} for details.

More precisely, let $(\xi_u, N_u =\sum_i\delta_ {(s_{ui}, y_{ui})})_{u\in \mathbb U}$ be the L\'evy processes and reproduction point processes sampled in the construction of $(f_u, \eta_u)_{u \in \mathbb{U}}$. Note that $f_u$ is the Lamperti transform of $\xi_u$ and $\eta_u$ is the image of $N_u$ under the time change. We identify the global time parameter of the system with that of the ultimate ancestor $(\xi_\varnothing, N_{\varnothing})$. For each $u\in\mathbb{U}^*$, we set $b(u) = \sum^{|u|}_{k=1} s_{u_k}$, which is the birth time of $(\xi_u, N_u)$ in the system. Define $\mathcal{N}_t = \{u: b(u)\leq t<z_u+b(u)\}$ to be the set of individuals alive at time $t$. For each $u\in \mathcal{N}_t$, we set its position $\Xi(t,u) = \xi_u(t-b(u))$. For ease of notation, if $u \notin \mathcal{N}_t$ but has an ancestor $v \in \mathcal{N}_t$ at time $t$, we set
\(\Xi(t,u) := \Xi(t,v)\). The system $(\Xi(t,u), u\in\mathcal{N}_t)_{t\geq 0}$ is then a branching L\'evy process.

The cumulant function $\kappa$ defined in \eqref{def: kappa} is also the Laplace exponent of this underlying branching L\'evy processes $(\Xi(t,u), u\in \mathcal{N}_t)$ in the sense that
\begin{equation}\label{eq: brl cumulant}
\mathbb{E}\bigg[\sum_{u\in\mathcal{N}_t} \exp(\gamma \Xi(t,u))\bigg] = \exp(t\kappa(\gamma)).
\end{equation}

\bigskip

\noindent \textbf{Many-to-one formula for branching L\'evy processes. } As a widely used tool for branching L\'evy processes, the many-to-one formula could convert calculations over all  individuals in the system to that on a single individual. We first introduce the L\'evy process $\hat{\xi}$ governing a tagged particle which will appear many times in this work. In \eqref{eq: brl cumulant}, by  \cite[Lemma 5.5]{BertoinJean2024SMta} the function \begin{equation}
\gamma \mapsto \kappa(\omega_-+\gamma) \label{eq: brl spine} \end{equation} is the L\'evy--Khintchine exponent of a L\'evy process $\hat{\xi}$ with explicit characteristics $(\mathrm{a}_{\omega_-}, \sigma^2, \Pi)$ (The precise definition of $\mathrm{a}_{\omega_-}$ and $\Pi$ will be given in Section \ref{sec: statement spinal decomposition}).  Write $\hat{P}$ and $\hat{E}$ for the probability and expectation of $\hat{\xi}$ which starts from $0$. We now import the many-to-one formula from \cite[Lemma 3.1]{MalleinBastien2023Anas}.
\begin{lem}\label{lem: brl many-to-one}
For a generic positive functional $F$, we have for $t\geq 0$,
\begin{equation}\label{eq: brl many-to-one}
\mathbb{E}\bigg[\sum_{u\in\mathcal{N}_t} F(\Xi(t,u))\bigg] = \hat{E}\bigg[\mathrm{e}^{-\omega_- \hat{\xi}_t}F(\hat\xi_t)\bigg].
\end{equation}

\end{lem}

\noindent We should also use the stopping line version. Roughly speaking, the stopping line is the collection of some stopping time along each lineage. We refer to \cite{BigginsJ.D.2004Mcim} for details.

\bigskip

\noindent \textbf{Many-to-one formula for branching random walks.} We shall also use the discrete many-to-one formula for branching random walk. Roughly speaking, a branching random walk is a discrete time particle system where in each generation, independently the individuals give birth to their offspring who have displacement from the position of the parent according to a certain point process. We refer to \cite{ShiZhan2016BRWE} for a detailed discussion. In our setting, the logarithms of types $(\log \chi(u))_{u\in \mathbb U}$ forms a branching random walk. Define the law of a random variable ${\hat{\tt S}}_1$ from the size-biasing: for all bounded and measurable functions $F$,
\[{\tt E}[F({\hat{\tt S}}_1)] = \mathbb{E}\bigg[\sum_{|u| = 1} (\chi(u))^{\omega_- } F(-\log \chi(u))\bigg].\]
Let ${\hat{\tt S}}_0 = 0$ and $({\hat{\tt S}}_i-{\hat{\tt S}}_{i-1})_{i\geq 1}$ be i.i.d random variables. We could now import the many-to-one formula from {\cite[Theorem 1.1]{ShiZhan2016BRWE}}

\begin{lem}\label{lem: many-to-one}
Fix $n \geq 1$, for a generic positive functional $F$ we have
\begin{equation}\label{eq: many-to-one}
\mathbb{E}\bigg[\sum_{|u|=n} (\chi(u))^{\omega_-}F(-\log\chi(u_1), \ldots, -\log\chi(u_n))\bigg] = {\tt E}[F({\hat{\tt S}}_1, \ldots, {\hat{\tt S}}_n)].
\end{equation}
\end{lem}

We will also use the stopping line version of \cref{lem: many-to-one}.

\subsection{Fluctuation theory for the L\'evy process $\hat{\xi}$}
\label{sec: fluctuations}
We collect here information about fluctuation theory for the process $\hat{\xi}$. See \cite{bertoin1996levy} for details.  Recall that we work under \cref{Assumption A} and let us introduce the renewal functions for the process ${\hat{\tt S}}$ and $\hat{\xi}$ in the discrete and continuous many-to-one formulas. The conditions $\kappa'(\omega_-) = 0$ and $\kappa''(\omega_-) < \infty$ implies that ${\tt E}[{\hat{\tt S}}_1] = 0$, ${\tt E}[{\hat{\tt S}}^2_1] < \infty$, $\hat{E}[\hat{\xi}_1] = 0$ and $\hat{E}[\hat{\xi}^2_1]<\infty$.  In particular, ${\hat{\tt S}}$ and $\hat{\xi}$ are oscillating. They even admit small exponential moments $ {\tt E}[\exp((\gamma-\omega_-){\hat{\tt S}}_1)]< \infty$ and  ${\hat E}[\exp((\gamma-\omega_-)\hat{\xi}_t)]< \infty$ for  $\gamma\in (\omega_-,\gamma_1)$.

For $b > 0$, let $\hat{\tau}_b = \inf\{t\geq 0: \hat{\xi}_t>b\}$ be the hitting time of $(b, \infty)$ and $\hat{\tau}_{-b}^- = \inf\{t\geq 0: \hat{\xi}_t<-b\}$ be the hitting time of $(-\infty, -b)$. Let $L$ be the local time at $0$ of the process $(\sup_{s\leq t} \hat{\xi}_s - \hat{\xi}_t)_{t\geq 0}$, and $L^-$ be the local time at $0$ of the process $(\hat{\xi}_t-\inf_{s\leq t} \hat{\xi}_s)_{t\geq 0}$. We choose an arbitrary normalisation of the local times, since the particular choice does not affect the results that follow. For $x\geq 0$, define the renewal functions
\begin{align}\label{eq: def R levy}
    R(b) = \hat{E}\bigg[\int^{\infty}_{0} \mathds{1}_{\{\hat{\tau}_b > t\}}\mathrm{d}L_t\bigg], \quad R^-(b) = \hat{E}\bigg[\int^{\infty}_{0}\mathds{1}_{\{\hat{\tau}^-_{-b} > t\}}\D L^-_t\bigg].
\end{align}

\noindent These functions are increasing and since $\hat{\xi}$ does not drift to $\infty$, by \cite[Lemma 1]{ChaumontLoïc2005OLpc} the function $R$ is harmonic for $\hat{\xi}$ killed above $b$, that is for $b > 0$ and $t\geq 0$
\begin{align}\label{eq: renewal martingale}
R(b) = \hat{E}[R(b-\hat{\xi}_t)\mathds{1}_{\{\hat{\tau}_b>t\}}].
\end{align}

\noindent We use the convention that $R(b) = R^-(b) = 0$ for $b<0$. The strong Markov property shows that $R$ and $R^-$ are sub-additive, which implies that they grow linearly with constants $c_0, c_0^->0$ (which depend on the normalisation of the local times) i.e. $R(b) \sim c_0 b$ and $R^{-}(b) \sim c_0^- b$ as $b \to \infty$. More precisely, the renewal theorem implies that for $h>0$ we have
\begin{equation}\label{eq: renewal theorem}
\lim_{x\to\infty} \frac{R([x, x+h])}{h} = c_0, \quad \lim_{x\to\infty} \frac{R^-([x, x+h])}{h} = c^-_0,
\end{equation}

\noindent where we let $R([x,x+h])=R(x+h)-R(x)$ and $R^-([x,x+h])=R^-(x+h)-R^-(x)$.

We define the ladder height processes $H(t)=\hat{\xi}_{L^{-1}(t)}$ and $H^{-1}(t)=\hat{\xi}_{(L^-)^{-1}(t)} $. Introduce the Laplace exponent of the ascending ladder process $(L^{-1},H)$, formally given by
\begin{align}\label{eq: def kappa+}
    \exp(-\kappa^+(a,b)) = \hat E_1[\exp(-a L^{-1}(1)-b H(1))].
\end{align}
In the same way, the Laplace exponent of the descending ladder process $((L^-)^{-1}, H^{-})$ is given by
\begin{align}\label{eq: def kappa-}
    \exp(-\kappa^-(a, b)) = \hat{E}_1[\exp(-a (L^-)^{-1}(1)- b H^-(1))].
\end{align}
By \cite[Section 5, Equation (3)(4)]{bertoin1996levy}, there exists a constant $K>0$ (which depend on the normalisation of the local times), such that  for $\lambda>0$ we have
\begin{equation}\label{eq: Hopf decomp}
\kappa^+(\lambda, 0) \kappa^-(\lambda, 0) = K\lambda
\end{equation}
and
\begin{equation*}
    K\psi_{\omega_-}(\lambda) = \kappa^+(0, -i\lambda)\kappa^-(0, i\lambda).
\end{equation*}
By the renewal theorem, we have $c_0 = \hat{\mathbb{E}}[H(1)]^{-1}$ and $c_0^- = \hat{\mathbb{E}}[H^-(1)]^{-1}$. The characteristic functions satisfy $\kappa^+(0, -i\lambda) \sim -i\lambda \hat{\mathbb{E}}[H(1)]$ and $\kappa^-(0, i\lambda) \sim i\lambda \hat{\mathbb{E}}[H^-(1)]$ when $\lambda$ is small. Together with $\psi_{\omega_-}(\lambda) = \kappa(\omega_-+\lambda)\sim \frac{\kappa^{\prime\prime}(\omega_-)\lambda^2}{2}$, we obtain a relation between $c_0$, $c_0^-$ and $K$:
\begin{equation}\label{eq: relation between K and c}
    Kc_0c_0^-=\frac{2}{\kappa''(\omega_-)}.
\end{equation}
We also introduce a variant of the renewal function $R$. For $\varepsilon>0$, we define $\mathscr{V}^{\varepsilon}: \mathbb R^+\to \mathbb R^+$ by
\begin{equation}\label{eq: def V}
   \mathscr{V}^{\varepsilon}(y):=\int_0^{\infty} \hat E_1[\exp(-\varepsilon L^{-1}(t))\cdot \mathds{1}_{\{H(t)\le y\}}]\D t.
\end{equation}
We see that $\mathscr{V}^{\varepsilon}(y)\uparrow R(y)$ as $\varepsilon\downarrow 0$.

For future use, we introduce a bit of the analogous results for the random walk ${\tt S}$. Similarly, we define the weak descend ladder process $({\tt T}_k, {\tt H}_k)$ and the strict ascending ladder process $({\tt T}^+_k, {\tt H}^+_k)$ of the random walk $-{\hat{\tt S}}$. We also define the discrete renewal functions ${\tt R}$ and ${\tt R}^+$ by
\begin{equation}\label{eq: def R rw}
{\tt R}(b) = {\tt E}\bigg[\sum^{\infty}_{k=0} \mathds{1}_{\{{\tt H}_k < b\}}\bigg], \quad {\tt R}^+(b) = {\tt E}\bigg[\sum^{\infty}_{k=0} \mathds{1}_{\{{\tt H}^+_k < b\}} \bigg].
\end{equation}
\noindent By convention, ${\tt R}(b) = {\tt R}^-(b) = 0$ for $b<0$. The renewal theorem states that there exist constants ${\tt c}_0, {\tt c}_0^- > 0$ such that
\begin{equation}\label{eq: def tt c_0}
\lim_{b\to\infty} \frac{{\tt R}(b)}{b} = {\tt c}_0, \quad \lim_{b\to\infty} \frac{{\tt R}^-(b)}{b} = {\tt c}^-_0.
\end{equation}

\subsection{Length measure of the tree below a barrier}\label{sec: tree below barrier}
As we have seen in \cref{sec: length measure}, the arguments of \cite[Section 2.3]{BertoinJean2024SMta} based on first moment calculation break down in the critical case. This comes from the unusual event that the decoration on the tree reaches a high level. The solution is to impose a barrier constraint on the decoration on the tree. This method is standard for the branching random walks in the boundary case. See for example \cite{ChenXinxin2015Anas, AidekonElie2024Bodt}.

Specifically, if $\mathtt{T}$ is a decorated tree starting from $g(\rho)=x$ then for $c\geq x$, we let $\mathtt{T}^c = (T^c, d_{T^c}, \rho, g)$ be the decorated subtree obtained by pruning the tree at the first time the decoration exceeds the level $c$, namely $$T^{c} = \{x\in T: g(\llbracket \rho, x \rrbracket)\leq c\}$$ where $\llbracket\rho,x\rrbracket$ is the geodesic from $\rho$ to $x$. Since $(||f_u||)_{u\in T}$ is a null family, the tree $T^c$ is obtained by pruning $T$ at finitely many points. In particular $T^c$ is closed a.s.. We show that $\mathtt{T}^c = \mathtt{T}$ with high probability as a foundation to the trick.

\begin{lem}\label{lem: Truncation}
Under \cref{Assumption A}, we have
\begin{equation}\label{eq: Truncation}
\mathbb{Q}_1(\mathtt{T}^c = \mathtt{T}) \geq 1 - c^{-\omega_-}.
\end{equation}
\end{lem}

\begin{proof}
Let $\mathcal{L}^c$ be the set of pairs $(u,t_u)$ in the branching L\'evy process $(\Xi(t,u), u\in \mathcal{N}_t)_{t\ge 0}$, where $t_u<\infty$ denotes the first time at which the process along the lineage hits $(\log c,\infty)$. Then $\mathcal{L}^c$ is a stopping line. By the many-to-one formula \eqref{eq: many-to-one}, we have
\[\mathbb{P}(\#\mathcal{L}^c\geq 1)\leq \mathbb{E}\Big[\sum_{(u, t_u)\in \mathcal{L}^c} 1\Big]\leq c^{-\omega_-}\mathbb{E}\Big[\sum_{(u, t_u)\in \mathcal{L}^c} \mathrm{e}^{\omega_- \Xi(t_u,u)}\Big] \stackrel{\mathrm{Lem.}\ref{lem: brl many-to-one}}{=} c^{-\omega_-}\hat{P}(\hat{\tau}_{\log c} < \infty) \leq  c^{-\omega_-}. \]
\end{proof}
Denote the expected total mass of the $\gamma$-length measure $\lambda^\gamma$ on $\mathtt{T}^c$ by $R_{\gamma}(c)$, that is,
\begin{equation}
R_{\gamma}(c) := \mathbb{E}\bigg[\int_{T^{c}} g(v)^{\gamma-\alpha}\D \lambda_{T}(v) \bigg].
\end{equation}
\noindent Recall the constant $K$ introduced in \eqref{eq: Hopf decomp}. We prove that $R_{\gamma}(c)$ is finite for $\gamma>\omega_-$:

\begin{prop}[$\gamma$-Length measure of the tree below a barrier]\label{prop: expected length}
Under \cref{Assumption A}, for $c>1$ and $\gamma>\omega_-$, we have
\begin{equation}\label{eq: expected length}
R_{\gamma}(c) = K\bigg(\int_{[0, \log c]}\mathrm{e}^{(\gamma-\omega_-)x}R(\D x)\bigg)\bigg(\int_{[0,\infty)}\mathrm{e}^{-(\gamma-\omega_-)x}R^-(\D x)\bigg)<\infty.
\end{equation}
\end{prop}

We deduce in particular that the measure $\lambda^{\gamma}$ is finite on $T^c$, $\mathbb{P}_1$-a.s for any $\gamma>\omega_-$. Since $\mathbb{Q}_1({\tt T}^c = {\tt T}) \to 1$ as $c\to\infty$ by \cref{lem: Truncation} we deduce that $\lambda^{\gamma}$ is a finite measure $\mathbb{P}_1$-a.s for any $\gamma > \omega_-$.

\begin{proof}
 Set $b = \log c$ and $\tau_{b}(u) = \inf\{t: \Xi(t,u) > b\}$. Recall that $\hat{\tau}_{b} = \inf\{t: \hat{\xi}_t > b\}$. We have
\[\mathbb{E}\bigg[\int_{T^{c}}g(v)^{\gamma-\alpha}\D \lambda_T(v)\bigg] \stackrel{\mathrm{Lamperti}}{=} \mathbb{E}\bigg[\int^{\infty}_{0}\bigg(\sum_{u\in {\cal N}_t} \mathrm{e}^{\gamma\Xi(t,u)}\cdot \mathds{1}_{\{\tau_b(u)>t\}} \bigg)\D t \bigg]  \stackrel{\mathrm{Lem.} \ref{lem: brl many-to-one}}{=} \hat{E}\bigg[\int^{\infty}_0 \mathrm{e}^{(\gamma-\omega_-)\hat{\xi}_t}\cdot\mathds{1}_{\{\hat{\tau}_b>t\}}\D t\bigg]. \]

\noindent By \cite[Lemma 20, Chapter VI]{bertoin1996levy}, for the constant $K>0$ introduced in \eqref{eq: Hopf decomp},
\[\hat{E}\bigg[\int^{\infty}_0 \mathrm{e}^{(\gamma-\omega_-)\hat{\xi}_t}\cdot \mathds{1}_{\{\hat{\tau}_b>t\}}\bigg] = K\bigg(\int_{[0, b]}\mathrm{e}^{(\gamma-\omega_-)x}R(\D x)\bigg)\bigg(\int_{[0,\infty)}\mathrm{e}^{-(\gamma-\omega_-)x}R^-(\D x)\bigg). \]

\noindent The limit \eqref{eq: renewal theorem} implies that the integrals on the right hand side are finite.
\end{proof}

\subsection{The harmonic measure}\label{sec: harmonic measure}

We now define the harmonic measure.  Under \cref{Assumption A}, since $\kappa(\omega_-) = 0$, the function $h(x) = x^{\omega_-}$ is a harmonic function (related to the additive martingale in the literature of branching random walks) i.e.
\begin{equation}\label{eq: harmonic function}
x^{\omega_-} = \mathbb{E}_x\bigg[\sum_{|u|=1} (\chi(u))^{\omega_-}\bigg], \quad \mbox{for all }x>0.
\end{equation}
This was used in \cite[Section 2.3.2]{BertoinJean2024SMta} to prove that the family of spreading mass
\begin{equation}\label{eq: spread mass subcritical}
\widetilde{m}_u = \lim_{n\to\infty} \sum_{|v|=n}(\chi(uv))^{\omega_-}
\end{equation}
 defines a random measure, called the harmonic measure on $ \mathtt{T}$. However, in the critical case, it is known, see \cite[Theorem 3.2]{ShiZhan2016BRWE}, that the total mass of this measure is $\widetilde{m}_\varnothing =0$. In order to define a non-degenerate harmonic measure, instead we should use here the derivative martingale
\begin{equation}
D_n = -\sum_{|u| = n} (\chi(u))^{\omega_-}\log(\chi(u)),
\end{equation} which also converges $\mathbb{P}$-a.s. to some random variable $D_{\infty}$ by \cite[Theorem 5.2]{ShiZhan2016BRWE}. To ensure positivity of the latter, we need to impose an analogue of \cite[Assumption 2.13]{BertoinJean2024SMta}:

\begin{Assump}[Critical Cram\'er condition]\label{Assumption B}
There exists $1<p_0\le 2$ and $0<\Delta_0<\gamma_1-\omega_-$ such that for $\gamma\in (\omega_--\Delta_0, \omega_- + \Delta_0)$,
\begin{align}\label{eq: assumption p moment}
\int_{\mathcal{S}_1} \mathbf{\Lambda}_1(\D y )\left(\sum_{i=1}^{\infty}\mathrm{e}^{y_i\gamma}\right)^{p_0}<\infty.
\end{align}
We also assume $\psi(\omega_-+\Delta_0)<0$ and $\psi(\omega_--\Delta_0)<0$ when $\Delta_0$ is small enough.
\end{Assump}

\begin{lem}[Definition of harmonic measure in the critical case] Under Assumptions \ref{Assumption A} and \ref{Assumption B}, under $\mathbb{P}_x$ for $x>0$, the family of spreading mass
\begin{equation}\label{eq: spread mass critical}
m_u = -\lim_{n\to\infty} \sum_{|v|=n}(\chi(uv))^{\omega_-}\log(\chi(uv))\ge 0,
\end{equation}
defines a non-trivial measure $\mu$ on the leaves of $\mathtt{T}$ called the \textbf{harmonic measure}.
\end{lem}
\begin{proof}
\noindent Under \cref{Assumption B}, for $1<p\leq p_0$ and $\gamma\in (\omega_--\Delta_0, \omega_- + \Delta_0)$, by \cite{IksanovAlexander2019Arop}, we have $\mathbb{E}_1[(\sum^{\infty}_{u=1}\chi(i)^{\gamma})^p]< \infty.$
\noindent By the convexity of $\psi$, there exist $1\le p<2$ and $C > 0$, for $\gamma\in  (\omega_--\Delta_0, \omega_- + \Delta_0)$,
\begin{equation}\label{eq: uniform p moment}
\psi(p\gamma)<0\quad \mbox{and}\quad \mathbb{E}\left[\left(\sum^{\infty}_{u=1}\chi(i)^{\gamma}\right)^{p}\right] \leq C<\infty.
\end{equation}

By \cite{ChenXinxin2015Anas}, we see that under \cref{Assumption A} and \cref{Assumption B}, $D_{\infty}$ is positive $\mathbb{P}_x$-a.s..

\noindent By \cite[Lemma 6.1]{buraczewski2021derivative}, we have a.s. $m_u=\sum_{i\geq 1} m_{ui}$. Therefore, by \cite[Section 1.3]{BertoinJean2024SMta}, the family of spreading mass $(m_u)_{u\in\mathbb{U}}$ gives rise to a harmonic measure $\mu$ supported on $\partial T$.

\end{proof}

As for the length measures, in general we have $\mathbb{E}[\mu(T)] = \mathbb{E}[D_{\infty}] = \infty$. We will see that when restricted to the truncated tree $\mathtt{T}_c$ the harmonic measure also has finite expectation. To this end, we introduce the truncated derivative martingale $(D^c_n)_n$, whose limit equals $D_{\infty}$ restricted on $T^c$ up to a multiplicative constant. Recall the renewal function $R$ defined in \cref{sec: fluctuations}.  For $u\in\mathbb{U}$, let $\rho_u$ be the root of the decorated branches $(f_u, \eta_u)$. For $c>x$, we set
\begin{equation}\label{eq: truncated derivative martingale}
D^c_n = \sum_{|u|=n} (\chi(u))^{\omega_-}\cdot  R(\log c-\log\chi(u))\cdot\mathds{1}_{\{g(\llbracket \rho, \rho_u \rrbracket)\leq c\}}.
\end{equation}

\begin{prop}[Harmonic measure of the tree below a barrier] \label{prop: derivative martingale}
For each $c>x$, we have $ \mathbb{E}_x\left[D^c_1\right] = x^{\omega_-}R(\log c-\log x)$ and thus $(D^c_n)_n$ is a non-negative martingale. As a consequence, the family of spreading mass
\begin{equation}\label{eq: spread mass critical truncated}
m_u^{(c)} = \mathds{1}_{\{X(\llbracket \rho, \rho_u \rrbracket)\leq c\}}\cdot\lim_{n\to\infty} \sum_{|v|=n} (\chi(uv))^{\omega_-} R(\log c-\log\chi(uv))\cdot\mathds{1}_{\{X(\llbracket \rho_u, \rho_{uv}\rrbracket)\leq c\}}
\end{equation}
defines a non-trivial measure $\mu^c$ on the leaves of $\mathtt{T}$ which is exactly $c_0 \mu$ restricted to $T^c$.
\end{prop}

\begin{proof}
Set $b = \log c$. We first prove $\mathbb{E}_x\left[D^c_1\right] = x^{\omega_-}R(\log c-\log x)$. Recall the branching L\'evy process $(\Xi(t,u), u\in\mathcal{N}_t)_{t\geq 0}$ constructed in \cref{sec: tools from branching processes} and  the hitting times $\tau_b(u) = \inf\{t: \Xi(t,u) > b\}$ for $u \in \mathbb{U}$. For $t>0$,
\begin{equation}\label{eq: renewal martingale in brl}
\mathbb{E}[\sum_{u\in \mathcal{N}_t} \mathrm{e}^{\omega_-\Xi(t,u))}R(b - \Xi(t,u)) \cdot \mathds{1}_{\{\tau_b(u) > t\}} ] \stackrel{\mathrm{Lem. }\ref{lem: brl many-to-one}}{=} x^{\omega_-}\hat{E}[R(b - \hat{\xi}_t) \cdot \mathds{1}_{\{\hat{\tau}_b > t\}}] \stackrel{\eqref{eq: renewal martingale}}{=} R(b).
\end{equation}
\noindent For $i\in\mathbb{N}^*$ we write $\mathcal{N}_t(i)$ for the set of $u\in\mathcal{N}_t$ such that $i$ is an ancestor of $u$ (possibly $i=u$) in the Ulam tree. We apply the display above to each $i$ if its birth time $b(i)\leq t$. Condition on $(\xi_{\varnothing}, N_{\varnothing})|_{[0, t]}$, we have
\begin{align*}
&\mathbb{E}\left[\sum_{u\in \mathcal{N}_t(i)} \mathrm{e}^{\omega_-\Xi(t,u))}R(b - \Xi(t,u)) \cdot \mathds{1}_{\{\tau_b(u) > t\}} \Big | (\xi_{\varnothing}, N_{\varnothing})|_{[0, t]}\right] \\
&\,=(\chi(i))^{\omega_-}R(b - \log(\chi(i)))\cdot \mathds{1}_{\{\tau_b(\varnothing)>b(i)\}} \cdot \mathds{1}_{\{b(i)\leq t\}}.
\end{align*}
Summing in $i$, we get
\begin{multline*}
\mathbb{E}\left[\sum_{u\in \mathcal{N}_t} \mathrm{e}^{\omega_-\Xi(t,u))}R(b - \Xi(t,u)) \cdot 1_{\{\tau_b(u) > t\}}  \Big | (\xi_{\varnothing}, N_{\varnothing})|_{[0, t]}\right]\\
= \mathrm{e}^{\omega_-\xi_{\varnothing}(t)}R(b -\xi_{\varnothing}(t))\cdot \mathds{1}_{\{\tau_b(\varnothing)>t\}} + \sum_{t_i < t} (\chi(i))^{\omega_-}R(b - \log(\chi(i)))\cdot \mathds{1}_{\{\tau_b(\varnothing)>b(i)\}}\cdot \mathds{1}_{\{b(i)\leq t\}}.
\end{multline*}
\noindent Taking expectation on both sides, the left hand side equals $x^{\omega_-}R(b-\log x)$ by \eqref{eq: renewal martingale in brl}. For the right hand side, we let $t\to\infty$ and show that the expectation of the summation increases to $\mathbb{E}[D_1^c]$. It suffices to show the expectation on the first term goes to $0$. Recall in \cref{Assumption A}, there exists $\gamma_1 > \omega_-$, such that $\psi(\gamma_1)<0$. By \eqref{eq: renewal theorem}, there exists $0<\delta<\gamma_1-\omega_-$ such that $R(x)\leq C \mathrm{e}^{\delta x}$. Thus
\[
E_x[\mathrm{e}^{\omega_-\xi_{\varnothing}(t)}R(b -\xi_{\varnothing}(t))\cdot \mathds{1}_{\{\tau_b(\varnothing)>t\}}] \leq CE_x[\mathrm{e}^{(\omega_- + \delta)\xi_{t}}] = C\mathrm{e}^{t\psi(\omega_- + \delta)} \xrightarrow[]{t\to\infty} 0
\]
\noindent since $\psi$ is convex and $\psi(\gamma)<0$ for $\gamma\in(\omega_-, \gamma_1)$.

It follows from the branching property and the above calculations that $(D^c_n)_{n\geq 0}$ is a non-negative martingale under $\mathbb{P}_x$ and thus converges a.s. Let us prove that the convergence is also in $L^1(\mathbb{P}_x)$. Using the asymptotics \eqref{eq: renewal theorem}, \eqref{eq: def tt c_0} and the fact ${\tt R}(0) > 0$, there exists a constant $C>0$ such that $R(x) \leq C{\tt R}(x)$. Therefore, $D^c_n$ is controlled by $C \sum_{|u| = n} (\chi(u))^{\omega_-} {\tt R}(\log c -\log(\chi(u)))\mathds{1}_{\{\chi(u_i)\leq c, i\leq |u|\}}$, which is uniformly integrable by \cite[Lemma 5.5]{ShiZhan2016BRWE}. Thus $D^c_n$ is uniformly integrable and thus converges in  $L^1(\mathbb{P}_x)$.

The mass $m^{(c)}_u$ in \eqref{eq: spread mass critical truncated} is now well-defined. By Fatou's lemma, $m^{(c)}_u\geq \sum_{i\geq 1} m^{(c)}_{ui}$. Taking expectation on both sides, we see the inequality is in fact an equality $\mathbb{P}_x$-a.s. since they have the same expectation. By \cite[Section 1.3]{BertoinJean2024SMta} again, the family of spreading mass $(m^{(c)}_u)_{u\in\mathbb{U}}$ induces a measure $\mu^c$ with $\mathbb{E}_x[\mu^c(T)]= x^{\omega_-}R(\log c - \log x) < \infty$ supported on $\partial T^c$. We call this measure $\mu^c$ the truncated harmonic measure.

We next show the measure $\mu^c$ and $c_0\mu$ restricted to $T^c$ are equal by a standard argument in branching random walks. By \cite[Lemma 3.1]{ShiZhan2016BRWE},  $M_n:=-\sup_{|u|=n}\log\chi(u)\to\infty$. For a fixed $c>0$, by \eqref{eq: renewal theorem}, $\mathbb{P}_1$-a.s., there exists $C>0$, for $y\geq 0$, $|R(y + \log c) - R(y)|\leq C|\log c|$ . As $n\to\infty$, we have
\begin{align*}
\sum_{|u|=n}(\chi(u))^{\omega_-}
\big|R(\log c-\log\chi(u))-R(-\log\chi(u))\big|\cdot
\mathds{1}_{\{X(\llbracket \rho,\rho_u \rrbracket)\le c\}}
\le C\,|\log c|\,W_n \longrightarrow 0
\end{align*}
and
\begin{align*}
\sum_{|u|=n}(\chi(u))^{\omega_-}
\big|R(-\log\chi(u))+c_0\log\chi(u)\big|
\cdot
\mathds{1}_{\{X(\llbracket \rho,\rho_u \rrbracket)\le c\}}
\le \sup_{x\ge M_n}
\Big|\frac{R(x)}{x}-c_0\Big|\, D_n
\longrightarrow 0.
\end{align*}
We see that the total variation between the families of spreading mass \eqref{eq: spread mass critical} and \eqref{eq: spread mass critical truncated} converges to $0$ when restricted to $T^c$, whence $\mu^c$ is the same as $c_0\mu$ restricted to $T^c$.

\end{proof}

\subsection{The Hausdorff dimension of the leaves}\label{sec: hausdorff dimension}

In this section, we compute the Hausdorff dimension of $\mathtt{T}$ as stated in \cref{thm: construct ssMt}. By self-similarity, it suffices to work under $\mathbb{Q}_1$. The approximation argument in \cref{sec: approximation} applies directly to the lower bound of Hausdorff dimension:

\begin{proof}[Proof of lower bound in \cref{thm: construct ssMt}. ]
Recall the construction of ${\tt T}^{i, \varepsilon}$ in \cref{prop: approximate critical from subcritical}. For ease of notation we fix $i = 2$ (the proof goes through for $i=1$ or $i = 3$ when $\Bd\Lambda$ is non-trivial) and omit the superscript $i$. We calculate that $\kappa^{\varepsilon}=\kappa-\epsilon$ by the explicit characteristic quadruplet in \cref{prop: approximate critical from subcritical}. When $\varepsilon>0$ is small enough, there exists $\omega^{\varepsilon}_-$ such that $\kappa^{\varepsilon}(\omega^{\varepsilon}_-) = 0$ and $\omega_-^{\varepsilon}\uparrow \omega_-$ as $\varepsilon\downarrow 0$. As a stronger assumption where we assume $L^p$ boundedness over an interval, \cref{Assumption B} implies \cite[Assumption 2.13]{BertoinJean2024SMta} ($L^p$ boundedness at a point). Applying \cite[Proposition 6.14]{BertoinJean2024SMta}, the Hausdorff dimension of $\partial{\tt T}^{\varepsilon}$ is $\omega_-^{\varepsilon}/\alpha$, $\mathbb{Q}_1$-a.s.. Since $p^{\varepsilon}: {\tt T}\to {\tt T}^{\varepsilon}$ is $1$-Lipschitz, for $d\in(0, \infty)$, the $d$-dimensional Hausdorff measure of $\partial T$ is no less than that of $\partial T^{\varepsilon}$. Thus $\dim_{H}(\partial T)\geq \omega_-^{\varepsilon}/\alpha$. The lower bound is established since $\omega_-^{\varepsilon}\uparrow \omega_-$ as $\varepsilon\downarrow 0$.
\end{proof}

In the subcritical case, the upper bound on the Hausdorff dimension is a consequence of $\mathrm{Height}(T)^{\gamma/\alpha}\in L^1(\mathbb{P})$ (see \cite[Lemma 2.6]{BertoinJean2024SMta}). We prove below the analogous estimate for the truncated tree in the critical case. The proof is inspired by the arguments of Aïdékon, Hu, and Shi \cite{AidekonElie2024Bodt}. The main idea is to estimate the local time of types staying inside intervals.

\begin{lem}\label{lem: height truncated tree}
Under \cref{Assumption A}, there exists a constant $C>0$ such that for any $c>0$ and any $\gamma\in(\omega_-, \gamma_1)$,
\begin{equation}
\mathbb{E}\Big[\mathrm{Height}(T^c)^{\gamma/\alpha}\Big]\leq C\cdot c^{(\gamma-\omega_-)}.
\end{equation}
\end{lem}

\begin{proof}
Let $\mathbb{U}^c = \{u\in\mathbb{U}: g(\llbracket \rho, \rho_u \rrbracket)\leq c\}$. Without loss of generality, we may assume that $c = \mathrm{e}^{b}$ where $b$ is an integer. By definition of $T^c$,
\[\mathrm{Height}(T^c) = \sup_{v\in T^c} d(\rho, v) \leq \sup_{u\in\mathbb{U}^c} \sum^{|u|}_{i=0} z_{u_i}. \]

\noindent For an integer $k\geq -b$ and $u\in\mathbb{U}$, define the local time of $ (\chi(u_i))_{i\le |u|}$ spent in the interval $(e^{-k-1},e^k)$ by
\[N^k_{u} = \sum^{|u|}_{i=0} \mathds{1}_{\{\mathrm{e}^{-k-1}< \chi(u_i)\leq \mathrm{e}^{-k}\}}.\]
\noindent For $n\geq 1$, define the stopping lines
\begin{equation}\label{eq: def stopping line Lkn}
\mathcal{L}^{k}_n = \{u\in\mathbb{U}: N^k_u = n, N^{k}_{u-} = n-1; \chi(u_i) \leq c,\forall 1\leq i\leq |u|\}.
\end{equation}
For every $u\in\mathbb{U}^c$, for each $0\leq i\leq |u|$ there exist unique $k$ and $n$ such that $u_i\in \mathcal{L}^{k}_n$. Hence
\begin{align}\label{eq: height as sum}
    \mathrm{Height}(T^c) \leq \sup_{u\in\mathbb{U}^c}
 \sum^{|u|}_{i=0} z_{u_i}\leq \sum^{\infty}_{k=-b}\sum^{\infty}_{n=1} \sup_{u\in\mathcal{L}^k_n} z_u.
\end{align}

\noindent By strong Markov property and then the self-similarity, for each $k$ and $n$,
\[\mathbb{E}\bigg[\sup_{u\in\mathcal{L}^k_n} z_u^{\gamma/\alpha}\bigg] \leq \mathbb{E}\bigg[\sum_{u\in\mathcal{L}^k_n} z_u^{\gamma/\alpha}\bigg] = \mathbb{E}\bigg[\sum_{u\in\mathcal{L}^k_n} E_{\chi(u)}[z_{\varnothing}^{\gamma/\alpha}]\bigg]\leq E_1[z_{\varnothing}^{\gamma/\alpha}]\mathbb{E}\bigg[\sum_{u\in\mathcal{L}^k_n}\chi(u)^{\gamma}\bigg]. \]
\noindent The technical lemma \cite[Lemma 9.1]{BertoinJean2024SMta} implies that $E[z_{\varnothing}^{\omega_-/\alpha}]\vee E[z_{\varnothing}^{\gamma_1/\alpha}]<\infty$. By Jensen's inequality, we have $C_1: = \sup_{\gamma\in(\omega_-, \gamma_1)}E[z_{\varnothing}^{\gamma/\alpha}]<\infty$. Define ${\tt T}^k_n = \inf\{m\ge 0: \Sigma^k_m\geq n, {\hat{\tt S}}_i\geq - b \mbox{ for } 0\leq i\leq m\}$ where $\Sigma^k_m := \sum^m_{i=0} \mathds{1}_{\{{\hat{\tt S}}_i \in [k, k+1)\}}$. By \cref{lem: many-to-one}, we have
\[\mathbb{E}\bigg[\sum_{u\in\mathcal{L}^k_n} \chi(u)^{\gamma}\bigg] = {\tt E}\bigg[\mathrm{e}^{-(\gamma-\omega_-){\hat{\tt S}}_{{\tt T}^k_n}}\cdot \mathds{1}_{\{{\tt T}^k_n < \infty\}}\bigg]\leq \mathrm{e}^{-(\gamma-\omega_-)k}{\tt P}({\tt T}^k_n < \infty). \]

\noindent We then control the probability ${\tt P}({\tt T}^k_n < \infty)$. Set ${\tt T}^-_{b} = \inf\{n: {\hat{\tt S}}_n \leq -b\}$ and ${\tt T}^+_{k} = \inf\{n: {\hat{\tt S}}_n \geq k\}$. By strong Markov property,
\[{\tt P}({\tt T}^k_n < \infty) \leq (1- \sup_{x\in [k, k+1)}{\tt P}_x({\tt T}^+_{k+1} > {\tt T}^-_{b}))^{n-1}. \]
By \cite[Theorem 5.1.7]{LawlerGregoryF.2010RWAM}, there exists a constant $C_2 > 0$ such that for $x\in (-b, a)$,
\[{\tt P}_x({\tt T}_b^- < {\tt T}_a^+) \geq C_2\frac{a-x+1}{b+a+1}.\]

\noindent Therefore, with $C(k) = C_2/(k+b+2)$,
\begin{equation}\label{eq: sum L k n}
\mathbb{E}\bigg[\sum_{u\in\mathcal{L}^k_n} \chi(u)^{\gamma}\bigg] \leq \mathrm{e}^{-(\gamma-\omega_-)k}(1-C(k))^{n-1}.
\end{equation}
We conclude that
\begin{equation}\label{eq: bound sup zu}
\mathbb{E}\bigg[\sup_{u\in\mathcal{L}^k_n} z_u^{\gamma/\alpha}\bigg]\leq C_1 \mathrm{e}^{-(\gamma-\omega_-)k}(1-C(k))^{n-1}.
\end{equation}

In the rest of the proof, we discuss under $\gamma/\alpha\leq 1$ and $\gamma/\alpha > 1$, respectively. We allow the constant $C$ to vary from line to line. In the case when $\gamma/\alpha\leq 1$, we have
\[\mathbb{E}\Big[\mathrm{Height}(T^c)^{\gamma/\alpha}\Big] \leq \mathbb{E}\bigg[\bigg(\sum^{\infty}_{k=-b}\sum^{\infty}_{n=1} \sup_{u\in\mathcal{L}^k_n} z_u\bigg)^{\gamma/\alpha}\bigg]\leq \sum^{\infty}_{k=-b}\sum^{\infty}_{n=1} \mathbb{E}\bigg[\sup_{u\in\mathcal{L}^k_n} z_u^{\gamma/\alpha}\bigg]. \]

\noindent By \eqref{eq: height as sum} and \eqref{eq: bound sup zu}, it ends up with
\[\mathbb{E}\Big[\mathrm{Height}(T^c)^{\gamma/\alpha}\Big] \leq C\sum^{\infty}_{k=-b}\sum^{\infty}_{n=1} \mathrm{e}^{-(\gamma-\omega_-)k}(1-C(k))^{n-1} = C\sum^{\infty}_{k=-b}\frac{1}{C(k)}\mathrm{e}^{-(\gamma-\omega_-)k} \leq C\cdot c^{\gamma-\omega_-}.\]

\noindent In the case when $\gamma/\alpha > 1$, we have
\begin{multline*}
\mathbb{E}\Big[\mathrm{Height}(T^c)^{\gamma/\alpha}\Big]^{\alpha/\gamma} \stackrel{\eqref{eq: height as sum}}{\leq} \mathbb{E}\bigg[\bigg(\sum^{\infty}_{k=-b}\sum^{\infty}_{n=1} \sup_{u\in\mathcal{L}^k_n} z_u\bigg)^{\gamma/\alpha}\bigg]^{\alpha/\gamma}
\stackrel{\mathrm{Minkovski}}{\leq} \sum^{\infty}_{k=-b}\sum^{\infty}_{n=1} \mathbb{E}\Big[\sup_{u\in\mathcal{L}^k_n} z_u^{\gamma/\alpha}\Big]^{\alpha/\gamma}\\
\stackrel{\eqref{eq: bound sup zu}}{\leq} C_1\sum^{\infty}_{k=-b}\sum^{\infty}_{n=1} \mathrm{e}^{-(\gamma-\omega_-)k\frac{\alpha}{\gamma}}(1-C(k))^{\frac{\alpha}{\gamma}(n-1)}
\leq  C\sum^{\infty}_{k=-b} (k+b+2)\mathrm{e}^{-(\gamma-\omega_-)k\frac{\alpha}{\gamma}} \leq C\cdot c^{(\gamma-\omega_-)\frac{\alpha}{\gamma}}.
\end{multline*}

\end{proof}

We can now provide the proof of the upper bound of Hausdorff dimension.

\begin{proof}[Proof of the upper bound in \cref{thm: construct ssMt}. ]
Since $\mathbb{Q}_1({\tt T}^c = T) \to 1$ as $c\to \infty$ in \cref{lem: Truncation}, it suffices to show that the Hausdorff dimension of $\partial T\cap T^c$ is smaller than $\omega_-/\alpha$, a.s. For $u\in \mathbb{U}^c$, define $T_u$ as the subtree of $T$ by gluing the decorated branches above $u$ in $\mathbb{U}$, and $T_u^c$ as the truncation of $T_u$ at $c$. Write $\mathbb{U}^c_n = \mathbb{U}^c\cap \mathbb{N}^n$. Then for each $n\geq 1$, we have $\partial T\cap T^c\subset \cup_{u\in\mathbb{U}^c_n} T^c_u$.

Fix $\gamma \in(\omega_-, \gamma_1)$. For $n\geq 1$ and $|u| = n$, conditioned on $\mathcal{F}_n$, the variable  $\mathrm{Height}(T^c_u)$ has the same law as $(\chi(u))^{\alpha}\mathrm{Height}(T^{c/\chi(u)})$ under $\mathbb{P}$ by self-similarity. By \cref{lem: height truncated tree}, we have
\begin{align*}
\mathbb{E}\bigg[\sum_{u\in\mathbb{U}^c_n} \mathrm{Diam}(T^c_u)^{\gamma/\alpha}\bigg]\leq 2^{\gamma/\alpha}\mathbb{E}\bigg[\sum_{u\in\mathbb{U}^c_n} \mathrm{Height}(T^c_u)^{\gamma/\alpha}\bigg]\leq Cc^{\gamma-\omega_-}2^{\gamma/\alpha}\mathbb{E}\bigg[\sum_{u\in\mathbb{U}^c_n} (\chi(u))^{\omega_-}\bigg].
\end{align*}

\noindent By \cref{lem: many-to-one} and \cite[equation (A.7)]{ShiZhan2016BRWE},
\[\mathbb{E}\bigg[\sum_{u\in\mathbb{U}^c_n} (\chi(u))^{\omega_-}\bigg] = {\tt P}({\hat{\tt S}}_i\geq -b, 1\leq i\leq n) = O\Big(\frac{1}{\sqrt{n}}\Big). \]

\noindent Summing over $n = 2^k$, we have
\[\sum^{\infty}_{k=0} \mathbb{E}\bigg[\sum_{u\in\mathbb{U}^c_{2^k}} (\mathrm{Diam}(T^c_u))^{\gamma/\alpha}\bigg]\leq Cc^{\gamma-\omega_-}\sum^{\infty}_{k=0} 2^{-k/2}<\infty. \]

\noindent Hence $\mathbb{P}$-a.s., as $k\to\infty$, we have $\sum_{u\in\mathbb{U}^c_{2^k}} (\mathrm{Diam}(T^c_u))^{\gamma/\alpha}\to 0$ by the Borel-Cantelli lemma. Since $\gamma\in(\omega_-, \gamma_1)$ is arbitrary, the Hausdorff dimension of $\partial T \cup T^c$ is bounded  above by $\omega_-/\alpha$. \end{proof}




\section{Properties of critical ssMt}\label{sec: Properties ssMt}

In this section we construct the spinal decomposition of critical ssMt and discuss bifurcators by importing results from the subcritical case using the approximating principle Proposition \ref{prop: approximate critical from subcritical}. Then we prove the convergence of $\gamma$-length measures to the harmonic measure on the critical ssMt (\cref{thm: weak convergence length measure}). This convergence is much more delicate to establish compared with the subcritical case and we need to develop delicate estimates on the truncated trees.

\subsection{Spinal Decomposition and bifurcators}\label{sec: spinal decomposition}
The spinal decomposition describes the size-biased law of the ssMt  with a marked point sampled from some measure on the tree (length or harmonic measure). Unlike the subcritical case where the total mass of the harmonic measure has finite expectation, size-biasing by $\mu$ is not well-posed in the critical case. To this end, we need to consider the spinal decomposition with respect to the truncated harmonic measure $\mu^c = c_0\mu|_{\mathtt{T}^c}$.

\subsubsection{The spinal decomposition theorem}\label{sec: statement spinal decomposition}
As expected, the evolution of decoration along the spine of a point sampled according to $\mu^c$ will involve a  L\'evy process conditioned to stay below the barrier $\log c$. We begin by constructing such processes.

\bigskip

\noindent \textbf{L\'evy process conditioned to stay below a barrier. } We introduce the measure $\boldsymbol{\Lambda}_{\omega_-}$ on $\cal{S}$ as in \cite[Section 5.2]{BertoinJean2024SMta}. Fix $i\geq 1$. For any pair $(y, \boldsymbol{y}) = (y, (y_\ell))_{\ell\geq 1}\in\mathbb{R}\times \mathcal{S}_1$, let $(y, \boldsymbol{y})^{\backsim i} \in \mathbb{R}\times \mathcal{S}_1$ be the pair $(y_i, \boldsymbol{y'})$ where $\boldsymbol{y'}$ is the re-ordering of $(y_1, \ldots, y_{i-1}, y, y_{i+1}, \ldots)$ in the non-increasing order. The measure $\boldsymbol{\Lambda}^{\backsim i}$ is then the push-forward of $\boldsymbol{\Lambda}$ by the map $(y, \boldsymbol{y})\mapsto (y, \boldsymbol{y})^{\backsim i}$. Define the measure $\boldsymbol{\Lambda}^{\backsim}_{\omega_-}$ by
\begin{equation}\label{eq: def lambda backsim}
\boldsymbol{\Lambda}^{\backsim}_{\omega_-}(\D y, \D \boldsymbol{y}) : = \mathrm{e}^{\omega_- y}\cdot \bigg(\sum_{i\geq 1}\boldsymbol{\Lambda}^{\backsim i}(\D y, \D \boldsymbol{y})\bigg).
\end{equation}

\noindent By definition of $\kappa(\gamma)$, we have
\begin{equation}\label{eq: finiteness mass}
\boldsymbol{\Lambda}^{\backsim}_{\omega_-}(\mathcal{S}) = \int_{\mathcal{S}}\bigg(\sum_{i\geq 1}\mathrm{e}^{\omega_- y_i}\bigg)\boldsymbol{\Lambda}(\D y, \D \boldsymbol{y}) = \kappa(\omega_-)-\psi(\omega_-) = -\psi(\omega_-)<\infty.
\end{equation}

\noindent Define
\begin{equation}\label{eq: def Lambda_gamma}
\boldsymbol{\Lambda}_{\omega_-}(\D y, \D \boldsymbol{y}) := \mathrm{e}^{\omega_- y}\cdot \boldsymbol{\Lambda}(\D y,\D \boldsymbol{y}) + \boldsymbol{\Lambda}^{\backsim}_{\omega_-}(\D y,\D \boldsymbol{y}).
\end{equation}

\noindent One could check that \eqref{eq: generalised Levy measure} still holds. Thus $\boldsymbol{\Lambda}_{\omega_-}$ is a generalised L\'evy measure\footnote{We remark that in the critical case, the definition \cite[Equation (5.12)]{BertoinJean2024SMta} does not give a generalised L\'evy measure when $\gamma>\omega_-$ (then $\kappa(\gamma)>0$) since the killing term $-\kappa(\gamma)\cdot\delta_{(-\infty,(0,-\infty,\ldots))}(\D y, \D \boldsymbol{y})$ is not well defined. This is why we focus on the harmonic measure only.} and the projection of $\boldsymbol{\Lambda}_{\omega_-}$ onto the first coordinate, $(\Bd\Lambda_{\omega_-})_0(\D x)$, is the same as the L\'evy measure $\Pi(\D x)$ in (\ref{eq: brl spine}). More precisely, the  function $\psi_{\omega_-}(q) := \kappa(\omega_-+q)$ can be expressed using \cite[Lemma 5.4]{BertoinJean2024SMta} as
\begin{equation}\label{def: psi}
\psi_{\omega_-}(q) = \frac{1}{2}\sigma^2q^2 + \mathrm{a}_{\omega_-}q + \int_{\mathcal{S}} (\mathrm{e}^{qy}-1-qy\mathds{1}_{\{|y|\leq 1\}})(\Bd{\Lambda}_{\omega_-})_0 (\mathrm{d}y, \mathrm{d}\boldsymbol{y}),
\end{equation}
\noindent where drift coefficient $\mathrm{a}_{\omega_-}$ is given by
\begin{equation}\label{eq: a gamma}
\mathrm{a}_{\omega_-} = \mathrm{a} + \frac{1}{2}\sigma^2\omega_-^2 + \int_{\mathcal{S}} \Big(y(\mathrm{e}^{\omega_- y}-1)\mathds{1}_{\{|y|\leq 1\}} + \sum^{\infty}_{i=1} y_i\mathrm{e}^{\omega_- y_i} \mathds{1}_{\{|y_i|\leq 1\}}\Big)\Bd\Lambda(\D y, \D\Bd y).
\end{equation}

\noindent As in \eqref{eq: brl spine} one sees that $\psi_{\omega_-}(q)$ is the Laplace exponent of the  L\'evy process $\hat{\xi}$. By an abuse of notation, we also denote by $\hat{P}_x$ and $\hat{E}_x $ the law and expectation of a decoration-reproduction process $(X,\eta)$ starting from $x$ with characteristic quadruplet $(\sigma^2, \mathrm{a}_{\omega_-}, \Bd\Lambda_{\omega_-};\alpha)$. We also simplify $\hat{P}_1$ and $\hat{E}_1$ as $\hat{P}$ and $\hat{E}$.

We proceed as in \cite[Section 3.2]{MalleinBastien2023Anas} to condition $(X,\eta)$ to stay below a barrier $ \mathrm{e}^b$: Recall from \eqref{eq: renewal martingale} that the renewal function $R$ is non-negative harmonic for $\hat{\xi}$ killed when going above $b$. One can then use this function to perform the $h$-transform, namely bias the law of $(\hat{\xi}, \hat{N})$ by $\frac{R(b-\hat{\xi}_t)}{R(b)}\mathds{1}_{\{\hat{\tau}_{b}>t\}}$ and then perform the Lamperti transform. The obtained law of decoration-reproduction process started from $x$, with characteristics $(\sigma^2, \mathrm{a}_{\omega_-}, \Bd\Lambda_{\omega_-};\alpha)$ and conditioned to stay below the barrier $ c=\mathrm{e}^b$ is denoted by $\hat{P}^{c}_x$.

\bigskip

\noindent \textbf{Statement of the spinal decomposition theorem.}
The spinal decomposition theorem provides two ways of description of the same law of decorated tree $\mathtt{T}^\bullet \in \mathbb{T}^\bullet$ with a distinguished point. The first law $\hat{\mathbb{Q}}^c_x$ is obtained by gluing dangling trees to the above spine: Sample the spine $([0, z], d, 0, \hat{X}^c)$ from $\hat{P}^c_x$ with the reproduction process $\hat{\eta}^c$ whose atoms are denoted  by $(t_i, y_i)_i$. Sample then a sequence of independent ssMt ${\tt T}_i = (T_i, d_{T_i}, \rho_i, g_i)$ (also independent of the spine) with characteristic exponent $(\sigma^2, \mathrm{a}, \boldsymbol{\Lambda}; \alpha)$ under the law $\mathbb{Q}_{y_i}$,  respectively. We then get a pointed decorated tree $\hat{\mathtt{T}}$ by gluing those trees on the points $(t_i)_i$ of the spine (provided that the resulting gluing is compact), see \cite{BertoinJean2024SMta} for details. The distinguished point of this tree lies at the extremity of the spine. The second law is obtained by size-biasing $\mathbb{Q}_x$ and sampling a point according to $\mu^c$. More precisely, for a generic function $F$ depending on a pointed decorated tree $\mathtt{T}^\bullet$, recalling Lemma \ref{prop: derivative martingale} we set
\[\tilde{\mathbb{Q}}^{c}_{x}[F({\tt T}^{\bullet})] = x^{-\omega_-}R(\log c-\log x)^{-1}\mathbb{Q}_{x}\bigg[\int_{T}F({\tt T}, r)\mu^c(\mathrm{d}r)\bigg].\]

\begin{thm}[Spinal decomposition]\label{thm: spinal decomposition truncated}
For $c>x>0$, the laws $\tilde{\mathbb{Q}}^c_x$ and $\hat{\mathbb{Q}}^c_x$ on $\mathbb{T}^\bullet$ are identical.
\end{thm}

\subsubsection{Proof of the spinal decomposition theorem.} \label{sec: proof spinal decomposition}
We apply the approximation arguments in \cref{sec: approximation} to prove the spinal decomposition theorem (\cref{thm: spinal decomposition truncated}) in this section. We always set $b=\log c$.

\begin{proof}[\textbf{Proof of \cref{thm: spinal decomposition truncated}}.]
We introduce the subcritical ssMt ${\tt T}^{\varepsilon}:= {\tt T}^{2, \varepsilon}$ coupled with ${\tt T} = (T, \rho, d, g)$ from \cref{prop: approximate critical from subcritical} by adding killing. The characteristic quadruplet of ${\tt T}^{\varepsilon}$ is  $(\sigma^2,\mathrm{a},\Bd{\Lambda}+\varepsilon\delta_{(-\infty;-\infty,\dots)};\alpha)$ and the cumulant function $\kappa^{\varepsilon}(\gamma)$ is $\kappa(\gamma)-\varepsilon$. Let $\omega_-^{\varepsilon}$ be the smallest zero of $\kappa^{\varepsilon}(\gamma) = 0$. By the construction of the coupling in \cref{sec: approximation}, we write ${\tt T}^{\varepsilon} = (T^{\varepsilon}, \rho, d^{\varepsilon}, g^\varepsilon)$ where $T^{\varepsilon}$ is a subtree of $T$, and $d^{\varepsilon}$, $g^{\varepsilon}$ denote the restriction of $d$ and $g$ on $T^{\varepsilon}$, respectively.

By self-similarity, we may assume $x=1$ without loss of generality. We apply the spinal decomposition for subcritical ssMt in \cite[Section 5]{BertoinJean2024SMta}. Let $\nu_{\varepsilon}$ be the length measure $\varepsilon\lambda^{\omega_-}$ restricted to $T^{\varepsilon}$. The total mass of $\nu_{\varepsilon}$ has expectation $\mathbb E[\nu_{\varepsilon}(T^{\varepsilon})]=1$ by \cite[(2.22)]{BertoinJean2024SMta}. Consider the probability measure $\tilde{\mathbb{Q}}_{(\varepsilon)}$\footnote{We add brackets to distinguish it from $\tilde{\mathbb{Q}}_{\varepsilon}$ which conventionally means the probability measure of a ssMt whose root has type $\varepsilon$. We also use notation $\tilde{\mathbb{Q}}^c_{(\varepsilon)}$ and $\hat{\mathbb{Q}}^c_{(\varepsilon)}$ below to distinguish from $\tilde{\mathbb{Q}}^c_{\varepsilon}$ and $\hat{\mathbb{Q}}^c_{\varepsilon}$.} on the space $\mathbb{T}^{\bullet}$ of pointed tree such that for any bounded measurable function $F:\mathbb{T}^{\bullet} \to \mathbb{R}$:
\begin{equation}\label{eq: size biase sub}
    \tilde{\mathbb{Q}}_{(\varepsilon)}(F({\tt T}^{\bullet}))= \mathbb Q\left(\int_{T^{\varepsilon}} F({\tt T}^{\varepsilon}, r)\nu_{\varepsilon}(\D r) \right).
\end{equation}
By the spinal decomposition for subcritical ssMt in \cite[Section 5]{BertoinJean2024SMta} this law is the same as the following one : let  $\hat{\mathbb{Q}}_{(\varepsilon)}$ be the law of a decorated tree obtained by first sampling a spine with law $\hat{P}^{\varepsilon}$ which is identical to $\hat{P}$ except that the underlying L\'evy process has an additional killing rate of $\varepsilon$. Given the atoms $(t_i, y_i)_i$ of the reproduction process, glue independent (also independent of the spine) ssMt ${\tt T}_i$ with the same law as ${\tt T}^{\varepsilon}$ under $\mathbb{Q}_{y_i}$.

We now want to condition on the spine staying below $c$. Let $\hat{\tau}_{b}=\inf\{t:\hat{\xi}(t)>b\}$. The probability for the spine to stay below $c$ is
\begin{align*}
\tilde{\mathbb{Q}}_{(\varepsilon)}\bigg(\mathds{1}_{\{g(\llbracket \rho, r \rrbracket)\leq c\}}\bigg) &= \hat{P}(\xi_t\leq \log c, 0\leq t\leq \zeta \wedge \mathrm{e}(\varepsilon))  \\
&= \varepsilon\int^{\infty}_0 \mathrm{e}^{-\varepsilon t}\hat{P}(\hat{\tau}_{\log c} > t) \D t.
\end{align*}
By \cite[Section 6, Equation (8)]{bertoin1996levy}, we have
\begin{equation*}
    \varepsilon\int_0^{\infty} \mathrm{e}^{-\varepsilon t}\hat P(\hat{\tau}_{\log c}>t)\D t=\kappa^+(\varepsilon,0)\mathscr{V}^{\varepsilon}(\log c)=:Z_{\varepsilon}(\log c).
\end{equation*}

\noindent Here $\kappa^+(a,b)$ and $ \mathscr{V}^{\varepsilon}$ are defined in \eqref{eq: def kappa+} and in \eqref{eq: def V}, respectively. As $\varepsilon$ goes to $0$, $\kappa^+(\varepsilon,0)$ decreases to $0$ and $\mathscr{V}^{\varepsilon}(\log c-\log x)$ increases to $R(\log c-\log x)$. Consider the conditional law $\tilde{\mathbb{Q}}^c_{(\varepsilon)}$ defined by
\begin{equation}
\tilde{\mathbb{Q}}^c_{(\varepsilon)}(F({\tt T}^{\bullet}))
 =  \frac{1}{Z_{\varepsilon}(\log c)}\mathbb Q\left(\int_{T^{\varepsilon}} F({\tt T}^{\varepsilon}, r)\cdot \mathds{1}_{\{g\llbracket \rho, r \rrbracket\leq c\}}\nu_{\varepsilon}(\D r) \right).
\end{equation}

Similarly, sample $(X,\eta)$ under $\hat{P}^{\varepsilon}$ conditioned on $X_t\leq c$ for $t\leq z^{\varepsilon}$ whose law is denoted by $\hat{P}^{c}_{1, \varepsilon}$ ($\hat{P}^{c}_{x, \varepsilon}$ for the process starting from $x$). Writing atoms of $\eta$ by $(t_i,y_i)_i$, we glue independent ssMt ${\tt T}_i$ with the same law as ${\tt T}^{\varepsilon}$ under $\mathbb{Q}_{y_i}$ at those points $(t_i)_i$. From \cite[Proposition 5.6]{BertoinJean2024SMta}, the gluing procedure is valid and we denote by $\hat{\mathbb{Q}}^c_{(\varepsilon)}$ be the law of the decorated tree obtained when $x=1$. Then we have $\hat{\mathbb{Q}}^c_{(\varepsilon)} = \tilde{\mathbb{Q}}^c_{(\varepsilon)}$ by  \cite[Proposition 5.7]{BertoinJean2024SMta}. It remains to show the following: 1) as $\varepsilon\to 0$, $\tilde{\mathbb{Q}}^c_{(\varepsilon)} \to \tilde{\mathbb{Q}}^c$; 2) as $\varepsilon\to 0$, $\hat{\mathbb{Q}}^c_{(\varepsilon)} \to \hat{\mathbb{Q}}^c$ and 3) the gluing under $\hat{\mathbb{Q}}$ produces compact decorated tree.

We first show that $\tilde{\mathbb{Q}}^c_{(\varepsilon)} \to \tilde{\mathbb{Q}}^c$. Take $n\in \mathbb N$ in the construction from the family of decoration-reproduction. We define ${\tt T}_n$ and ${\tt T}_n^{\varepsilon}$ be the subtree of $\tt T$ and $\tt T^{\varepsilon}$ truncated at generation $n$. We write $r$ as $(u^\bullet,t^\bullet)$, where $u^\bullet\in \mathbb U$ and $t^\bullet\in [0,z_{u^\bullet}]$. Let $\mathcal F_n$ be the $\sigma$-field generated by the pointed decorated tree up to generation $n$. Take a bounded functional $F_n$ measurable to $\mathcal{F}_n$. Since $\mathbb Q(\lambda^{\omega_-}({\tt T}_n))$ is finite and $F_n$ is bounded, we see by \eqref{eq: relation between K and c}
\begin{align*}
     \frac{ \mathbb Q\left(\int_T F_n({\tt T}^\varepsilon_n,u^{\bullet}_n)  \mathds{1}_{\{g(\llbracket \rho, r \rrbracket)\le c \}}\mathds{1}_{\{|u^\bullet|< n\}}\nu_{\varepsilon}(\D r) \right)}{Z_{\varepsilon}(\log c)}\le C \frac{\varepsilon}{\kappa^+(\varepsilon,0)} = KC\kappa^-(\varepsilon,0).
\end{align*}
where $C$ is a positive constant which does not depend on $\varepsilon$ and $\kappa^-(\varepsilon,0)$ goes to $0$ as $\varepsilon$ decreases to $0$. For each $|u|=n$, we have by taking conditional expectation with respect to $\mathcal{F}_n$ and applying Markov property
\begin{equation}\label{eq: prob to choose u}
\begin{split}
&\frac{1}{Z_{\varepsilon}(\log c)}\mathbb Q\left(\int_T F_n({\tt T}_n^\varepsilon,u^{\bullet}_n) \mathds{1}_{\{g(\llbracket \rho, r \rrbracket)\le c \}}\mathds{1}_{\{|u^\bullet|\ge n,u^\bullet_n=u\}}\nu_{\varepsilon}(\D r) \right)\\
&\,= \mathbb Q\left(F_n({\tt T}_n^\varepsilon,u)\mathds{1}_{\{g(\llbracket \rho,\rho_u \rrbracket)\le c\}}\frac{\mathbb Q_{\chi(u)}\left(\int_T \mathds{1}_{\{g(\llbracket \rho, r \rrbracket)\le c \}}\nu_{\varepsilon}(\D r) \right)}{Z_{\varepsilon}(\log c)}\right)\\
&\,= \mathbb Q\left(F_n({\tt T}_n^\varepsilon,u)\mathds{1}_{\{g(\llbracket \rho,\rho_u \rrbracket)\le c\}}(\chi(u))^{\omega_-}\frac{ Z_{\varepsilon}(\log c-\log \chi(u))}{Z_{\varepsilon}(\log c)}\right).
\end{split}
\end{equation}
As $\varepsilon$ goes to $0$, the ssMt ${\tt T}^{\varepsilon}_n$ converges to ${\tt T}_n$. Recall that $\mathscr{V}^{\varepsilon}(\log c)$ increases to $R(\log c)$. We see that, for each $y<\log c$
\begin{align}\label{eq: fractral V to R}
    \frac{Z_{\varepsilon}(\log c - y)}{Z_{\varepsilon}(\log c )} = \frac{\mathscr{V}^{\varepsilon}(\log c -  y)}{\mathscr{V}^{\varepsilon}(\log c )}\longrightarrow  \frac{R(\log c - y)}{R(\log c )}.
\end{align}
We see that $\mathscr{V}^{\varepsilon}(\log c-\log x) $ increases to $R(\log c-\log x)$ as $\varepsilon\to 0$. By \eqref{eq: fractral V to R} and the dominated convergence theorem, the last term in \eqref{eq: prob to choose u} converges to
\begin{multline*}
\mathbb Q\left( F_n({\tt T}_n,u)\mathds{1}_{\{g(\llbracket \rho,\rho_u \rrbracket)\le c\}}(\chi(u))^{\omega_-}\frac{R(\log c-\log \chi(u))}{R(\log c)}\right)\,=\\
R(\log c)^{-1}\mathbb Q\left(\int_T F_n({\tt T}_n,u^{\bullet}_n)\mathds{1}_{\{g(\llbracket \rho, r \rrbracket)\le c \}}\mathds{1}_{\{|u^\bullet|\ge n,u^\bullet_n=u\}}\mu^c{}(\D r) \right).
\end{multline*}
Then we conclude that $\tilde{\mathbb{Q}}^{c}_{(\varepsilon)}(A_n) \to \tilde{\mathbb{Q}}^c(A_n)$ for each $n\in \mathbb N$ and $A_n\in \mathcal{F}_n$. It implies $\tilde{\mathbb{Q}}^{c}_{(\varepsilon)} \to \tilde{\mathbb{Q}}^c$ by the monotone class theorem.

We then prove that $\hat{\mathbb{Q}}^{c}_{(\varepsilon)} \to \hat{\mathbb{Q}}^c$. To begin with, we discuss the law of the spine. For fixed $t>0$, and any bounded measurable functional on the space of processes on $[0, t]$,
\begin{align*}
\hat{E}^c_{1, \varepsilon}[F(\hat{\xi}|_{[0, t]})] &= (Z_{\varepsilon}(\log c ))^{-1}\hat{E}[F(\hat{\xi}|_{[0, t]})\cdot \mathds{1}_{\{\hat\xi_s \leq \log c, 0\leq s\leq \zeta\wedge \mathrm{e}(\varepsilon)\}}]\\
&= (Z_{\varepsilon}(\log c))^{-1}\hat{E}[F(\hat{\xi}|_{[0, t]})\cdot \mathds{1}_{\{\hat\xi_s \leq \log c, 0\leq s\leq \zeta\wedge \mathrm{e}(\varepsilon)\}}\mathds{1}_{\{\mathrm{e}(\varepsilon)>t \}}]\\
&\quad+ (Z_{\varepsilon}(\log c))^{-1}\hat{E}[F(\hat{\xi}|_{[0, t]})\cdot \mathds{1}_{\{\hat\xi_s \leq \log c, 0\leq s\leq \zeta\wedge \mathrm{e}(\varepsilon)\}}\mathds{1}_{\{\mathrm{e}(\varepsilon)\leq t \}}].
\end{align*}
The second term on the right hand side converges to $0$ as $\varepsilon\to 0$ since $F$ is bounded, $P(\mathrm{e}(\varepsilon)\leq t) = 1-\exp(-\varepsilon t) \sim t \varepsilon$ and $\varepsilon/\kappa^+(\varepsilon, 0) = K^{-1}\kappa^-(\varepsilon, 0)\to 0$. Conditioned on $\mathrm{e}(\varepsilon)>t$, $\mathrm{e}'(\varepsilon) := \mathrm{e}(\varepsilon)-t$ is exponentially distributed with parameter $1/\varepsilon$. By Markov property at time $t$, the first term equals
\begin{equation}\label{eq: convergence from Z to R}
\begin{split}
    &\hat{E}\bigg[F(\hat{\xi}|_{[0, t]})\cdot \mathds{1}_{\{\hat\xi_s \leq \log c, 0\leq s\leq t\wedge \zeta\}}\mathds{1}_{\{\mathrm{e}(\varepsilon) > t\}}\frac{\hat{P}^c_{y, \varepsilon}(\hat{\xi}_s\leq \log c, 0\leq s \leq \zeta\wedge\mathrm{e}'(\varepsilon))|_{y = \exp(\hat{\xi}_t)}}{Z_{\varepsilon}(\log c )}\bigg]\\
&\,= \hat{E}\bigg[F(\hat{\xi}|_{[0, t]})\cdot \mathds{1}_{\{\hat\xi_s \leq \log c, 0\leq s\leq t\wedge \zeta\}}\mathds{1}_{\{\mathrm{e}(\varepsilon) > t\}}\frac{Z_{\varepsilon}(\log c - \hat{\xi}_t)}{Z_{\varepsilon}(\log c )}\bigg].
\end{split}
\end{equation}
Then \eqref{eq: fractral V to R} implies that the expectation on the right hand side of \eqref{eq: convergence from Z to R} converges to
\[\hat{E}\bigg[F(\hat{\xi}|_{[0, t]})\cdot \mathds{1}_{\{\hat\xi_s \leq \log c, 0\leq s\leq t\wedge \zeta\}}\frac{R(\log c - \hat{\xi}_t)}{R(\log c )}\bigg]=\hat{E}^{c}[F(\hat{\xi}|_{[0, t]}) ]\]
as $\varepsilon\to 0$ by dominated convergence. Since $t$ is arbitrary, we conclude that the law of the spine satisfies $\hat{P}^c_{1,\varepsilon} \to \hat{P}^c$.

The law for each dangling tree ${\tt T}_i$ converges a.s. to a ssMt with law $\mathbb{P}_{y_i}$ by \cref{prop: approximate critical from subcritical}. To see the independence of the dangling trees, we remark that the type of the root $y_i$ does not depends on $\varepsilon$ once its birth time $t_i$ is smaller than $\mathrm{e}(\varepsilon)$. Thus the independence still holds as $\varepsilon\to 0$. This concludes that $\hat{\mathbb{Q}}^{c}_{(\varepsilon)} \to \hat{\mathbb{Q}}^c$.

Finally, to see that the gluing operation produces compact decorated trees, notice that the decorated trees under $\hat{\mathbb{Q}}^c_{(\varepsilon)}$ are subtrees of ${\tt T}^{\bullet}$ under $\tilde{\mathbb{Q}}^c$ by the coupling. Therefore, $\hat{\mathbb{Q}}^c_{(\varepsilon)} \to \hat{\mathbb{Q}}^c$ could be viewed as convergence of subtrees inside ${\tt T}^{\bullet}$ under $\tilde{\mathbb{Q}}^c$. Since ${\tt T}^{\bullet}$ is compact under $\tilde{\mathbb{Q}}^c$, the same holds for $\hat{\mathbb{Q}}^c$.
\end{proof}

\subsubsection{Bifurcators}\label{sec: bifurcators}

We next discuss the bifurcators in the critical case. Bifurcation is the phenomenon that two different characteristic quadruplets may determine the same law of a ssMt: two such quadruplets are called bifurcators of each other. We recall some background from \cite[Chapter 5]{BertoinJean2024SMta}: Let $\ord: \mathcal{S} \to \mathcal{S}_1$ be the map sending $(y_0, (y_i)_{i\geq 1})$  to the sequence $(y_i)_{i\geq 0}$ arranged in decreasing order.

\begin{defn}
We say that the two quadruplets $(\sigma^2, \mathrm{a}, \Bd\Lambda; \alpha)$ and $(\sigma^2_{\Yleft}, \mathrm{a}_{\Yleft}, \Bd\Lambda_{\Yleft}; \alpha_{\Yleft})$ are bifurcators of each other, denoted by $(\sigma^2, \mathrm{a}, \Bd\Lambda; \alpha)\approx(\sigma^2_{\Yleft}, \mathrm{a}_{\Yleft}, \Bd\Lambda_{\Yleft}; \alpha_{\Yleft})$, if and only if
\begin{equation}\label{eq: bifurcator condition 1}
\sigma^2 = \sigma^2_{\Yleft}, \quad \alpha = \alpha_{\Yleft},\quad \Bd\Lambda \circ \ord^{-1} = \Bd\Lambda_{\Yleft}\circ \ord^{-1},
\end{equation}
\noindent and
\begin{equation}\label{eq: bifurcator condition 2}
\mathrm{a} - \mathrm{a}_{\Yleft} = \lim_{\varepsilon\to 0+}\bigg(\int_{\varepsilon < |y|\leq 1} y\cdot\Bd\Lambda(\D y,\D \Bd y) - \int_{\varepsilon < |y|\leq 1}y\cdot \Bd\Lambda_{\Yleft}(\D y, \D\Bd y)\bigg).
\end{equation}

\end{defn}

Now we fix two quadruplets $(\sigma^2, \mathrm{a}, \Bd\Lambda; \alpha)$ and $(\sigma^2_{\Yleft}, \mathrm{a}_{\Yleft}, \Bd\Lambda_{\Yleft}; \alpha_{\Yleft})$. With these two quadruplets we define, respectively, $P_x$ and $P^{\Yleft}_x$ as the laws of decoration-reproduction processes for an individual with type $x>0$, $\mathbb{P}_x$ and $\mathbb{P}^{\Yleft}_x$ as the laws of the families of decoration-reproduction processes, and finally $\mathbb{Q}_x$ and $\mathbb{Q}^{\Yleft}_x$ as the laws of self-similar Markov trees with roots of type $x$.

\begin{thm}\label{thm: bifurcators}
Assume that $(\sigma^2, \mathrm{a}, \Bd\Lambda; \alpha)$ and $(\sigma^2_{\Yleft}, \mathrm{a}_{\Yleft}, \Bd\Lambda_{\Yleft}; \alpha_{\Yleft})$ satisfy \cref{Assumption A}. Then $\mathbb{Q}_x = \mathbb{Q}^{\Yleft}_x$ for all $x>0$, if and only if $(\sigma^2, \mathrm{a}, \Bd\Lambda; \alpha)$ and $(\sigma^2_{\Yleft}, \mathrm{a}_{\Yleft}, \Bd\Lambda_{\Yleft}; \alpha_{\Yleft})$ are bifurcators of each other.
\end{thm}

To prove the theorem, one could use the spinal decomposition theorem \cref{thm: spinal decomposition truncated} following the same arguments as in \cite[Section 5.3]{BertoinJean2024SMta}. We instead apply the approximation arguments in \cref{sec: approximation} to avoid repetition.

\begin{proof}
Let ${\tt T}^{\varepsilon}$ be the subcritical trees coupled with ${\tt T}$ as in \cref{sec: approximation} by adding a killing rate of $ \varepsilon$, and similarly ${\tt T}^{\varepsilon}_{\Yleft}$ for  ${\tt T}_{\Yleft}$. Suppose that $(\sigma^2, \mathrm{a}, \Bd\Lambda; \alpha)$ and $(\sigma^2_{\Yleft}, \mathrm{a}_{\Yleft}, \Bd\Lambda_{\Yleft}; \alpha_{\Yleft})$ are bifurcators of each other, then for each $\varepsilon > 0$, the pair $(\sigma^2, \mathrm{a}-\varepsilon, \Bd\Lambda; \alpha)$ and $(\sigma^2_{\Yleft}, \mathrm{a}_{\Yleft}-\varepsilon, \Bd\Lambda_{\Yleft}; \alpha_{\Yleft})$ are bifurcators of each other. By \cite[Theorem 5.13]{BertoinJean2024SMta}, the trees ${\tt T}^{ \varepsilon}$ and ${\tt T}^{\varepsilon}_{\Yleft}$ have the same law. By \cref{prop: approximate critical from subcritical} we have ${\tt T}^{\varepsilon}\to {\tt T}$ and ${\tt T}^{ \varepsilon}_{\Yleft} \to{\tt T}_{\Yleft}$ almost surely, we conclude that $\mathbb{Q}_x = \mathbb{Q}^{\Yleft}_x$.

Conversely, we assume that $\mathbb{Q}_x = \mathbb{Q}^{\Yleft}_x$. By self-similarity, we have $\alpha = \alpha_{\Yleft}$. From \cref{rem: intrinsic approx}, adding killing to a ssMt is intrinsic hence ${\tt T}^{ \varepsilon}$ and ${\tt T}^{ \varepsilon}_{\Yleft}$ have the same law. By \cite[Theorem 5.13]{BertoinJean2024SMta} again, we see that $(\sigma^2, \mathrm{a}-\varepsilon, \Bd\Lambda; \alpha)$ and $(\sigma^2_{\Yleft}, \mathrm{a}_{\Yleft}-\varepsilon, \Bd\Lambda_{\Yleft}; \alpha_{\Yleft})$ are bifurcators of each other. So are $(\sigma^2, \mathrm{a}, \Bd\Lambda; \alpha)$ and $(\sigma^2_{\Yleft}, \mathrm{a}_{\Yleft}, \Bd\Lambda_{\Yleft}; \alpha_{\Yleft})$. \end{proof}

\subsection{Convergence from length measures to the harmonic measure. }\label{sec: convergence between measures}

Recall from \cref{prop: expected length} that the $\gamma$-length measure $\lambda^\gamma$ is defined for $\gamma>\omega_-$. In the next result we show that the harmonic measure is a limit of the renormalised length measures (in particular, it is intrinsic):\begin{thm}\label{thm: weak convergence length measure} Under \cref{Assumption A} and \cref{Assumption B}, there exists a sequence $(\gamma_n)_n$ with $\gamma_n \downarrow\omega_-$ such that $\mathbb{P}_1$-a.s.,
\[\lim_{\gamma_n \downarrow\omega_-} \frac{\kappa^{\prime\prime}(\omega_-)}{2}(\gamma_n-\omega_-)\lambda^{\gamma_n} = \mu, \]

\noindent in the sense of weak convergence of finite measures on $T^c$.
\end{thm}

\begin{rem}
Here we restrict ourselves to convergence along a sequence as in \cite[Proposition 2.15]{BertoinJean2024SMta} though we also believe it also holds for $\gamma\downarrow\omega_- $.
\end{rem}

\noindent We denote by $\lambda^\gamma_c$ the measure $\lambda^\gamma$ restricted to $T_c$. We see ${\tt T}^c=\tt T$ a.s. for some $c$ by \cref{lem: Truncation} and $\frac{1}{c_0}\mu^c = \mu$ on ${\tt T}^c$ by \cref{prop: derivative martingale}. By the relation \eqref{eq: relation between K and c}, it suffices to show that $\frac{(\gamma - \omega_-)}{Kc_0^-}\lambda^{\gamma}_c$ a.s. converges to $\frac{1}{c_0}\mu^c$ along a sequence for $c = \log b$ with $b\in\mathbb{N}$, and then apply the standard diagonal argument. With \eqref{eq: relation between K and c}, \cref{thm: weak convergence length measure} is a direct consequence of the following lemma by adapting the  proof of  \cite[Proposition 2.15]{BertoinJean2024SMta}:

\begin{lem}\label{lem: convergence total mass}
We have
\[\lim_{\gamma\downarrow\omega_-}\frac{(\gamma-\omega_-)}{Kc_0^-}\lambda^{\gamma}_c(T^c) = \mu^c(T^c), \quad \mbox{in } L^1(\mathbb{P}_1). \]
\end{lem}

From now on we focus on the proof of \cref{lem: convergence total mass}. The idea of the proof resembles that in the subcritical case. We first introduce a measure $\tilde{\lambda}^{\gamma}_c$ obtained by transferring the mass of each decorated branch under $\D\lambda^{\gamma}_c$ to its root. The difference between $\tilde{\lambda}^{\gamma}_c$ and ${\lambda}^{\gamma}_c$ is small by a similar argument of claim (ii) in the proof of \cite[Lemma 2.16]{BertoinJean2024SMta}. The next step is the main technical part, which is to compare $\tilde{\lambda}^{\gamma}_c$ with its expectation conditioned on the information before a large generation. This conditional expectation converges to $\mu^c(T^c)$ as the generation goes to $\infty$. We now formalise these ideas.

Consider the contribution from a decorated branch with type $x<c$ to the total mass of $\lambda^{\gamma}_c$. The expectation of $\lambda^{\gamma}_c$ over this decorated branch is given by
\[E_x\bigg[\int^{z}_{0}X(s)^{\gamma-\alpha}\mathds{1}_{\{\tau^c>t\}}\D t\bigg] = x^sE\bigg[\int^{\zeta}_{0}\mathrm{e}^{\gamma\xi_s}\mathds{1}_{\{\tau^{\log(c/x)}>s\}}\D s\bigg]. \]
For ease of notation, we set
$$
r_{\gamma}(c) : = \mathbb{E}[\int^{\zeta}_{0}\mathrm{e}^{\gamma\xi_s}\mathds{1}_{\{\tau^{\log c}>s\}}\D s]
$$
\noindent We could rewrite \eqref{eq: expected length} as
\begin{equation}\label{eq: expected length 2}
R_{\gamma}(c) = \mathbb{E}[\lambda^{\gamma}_c(T^c)] = \mathbb{E}\bigg[\sum_{u\in\mathbb{U}}(\chi(u))^{\gamma}r_{\gamma}\Big((c/\chi(u))\Big)\cdot\mathds{1}_{\{g(\llbracket \rho, \rho_u \rrbracket)\leq c\}} \bigg].
\end{equation}

\noindent A rough bound on $r_{\gamma}(c)$ is
\begin{equation}\label{eq: rough bound r c}
r_{\gamma}(c)\leq \mathbb{E}\bigg[\int^{\zeta}_0 \mathrm{e}^{\gamma\xi_s}\D s \bigg] = -\frac{1}{\psi(\gamma)}
\end{equation}

We define the measure $\tilde{\lambda}^{\gamma}_c $ by

\begin{equation}\label{eq: def length measure condensed}
\tilde{\lambda}^{\gamma}_c = \sum_{u\in \mathbb{U}} \delta_{\rho_u} (\chi(u))^{\gamma} r_{\gamma}\Big(\frac{c}{\chi(u)}\Big)\cdot \mathds{1}_{\{g(\llbracket \rho,\rho_u \rrbracket)\leq c\}}.
\end{equation}

\noindent The next lemma shows that the difference between the total mass of $\tilde{\lambda}^{\gamma}_c$ and ${\lambda}^{\gamma}_c$ remains bounded. The proof of this lemma is close to (ii) in the proof of \cite[Lemma 2.16]{BertoinJean2024SMta}. We use the strategy in \cref{lem: height truncated tree} to sum over intervals and apply the estimates on the local time for the types.

\begin{lem}\label{lem: transfer length to mass}
There exists a constant $C>0$, such that for $\gamma \in (\omega_-, \omega_- + \Delta_0)$,
\begin{equation}\label{eq: transfer length to mass}
\mathbb{E}\Big[\Big|\tilde{\lambda}^{\gamma}_c(T^c)-\lambda^{\gamma}_c(T^c)\Big|\Big]\leq C.
\end{equation}
\end{lem}

\begin{proof}[Proof of \cref{lem: transfer length to mass}]
Write for $u\in\mathbb{U}$,
\[A^{c}_{\gamma}(u) = \chi(u)^{-\gamma}\int^{z_u}_0 f_u(t)^{\gamma-\alpha}\cdot\mathds{1}_{\{\tau^{\log(c/\chi(u))}>t\}}\mathrm{d}t. \]
\noindent We rewrite the difference of measure by (recall the definition of $\mathcal{L}^k_n$ in \eqref{eq: def stopping line Lkn})
\begin{eqnarray*}
\lambda^{\gamma}_c(T^c) - \tilde{\lambda}^{\gamma}_c(T^c) &=& \sum_{u\in\mathbb{U}}\mathds{1}_{\{g(\llbracket \rho, \rho_u \rrbracket)\leq c\}}\cdot \chi(u)^{\gamma}\bigg(A^{c}_{\gamma}(u) - r_{\gamma}\Big(\frac{c}{\chi(u)}\Big)\bigg) \\
&=& \sum_{k=-b}^{\infty}\sum^{\infty}_{n=1} \sum_{u\in\mathcal{L}^k_n} \mathds{1}_{\{g(\llbracket \rho, \rho_u \rrbracket)\leq c\}}\cdot \chi(u)^{\gamma} \bigg(A^{c}_{\gamma}(u) - r_{\gamma}\Big(\frac{c}{\chi(u)}\Big)\bigg).
\end{eqnarray*}

\noindent Choose $q\in (1, 2)$ to be defined later. By triangle inequality and the Jensen's inequality, we bound the expectation in \eqref{eq: transfer length to mass} by
\[ \mathbb{E}\Big[|\tilde\lambda^{\gamma}_c(T^c)-\lambda^{\gamma}_c(T^c)|\Big] \leq \sum^{\infty}_{k=-b}\sum^{\infty}_{n=1}\mathbb{E}\Bigg[\bigg|\sum_{u\in\mathcal{L}^k_n}\mathds{1}_{\{g(\llbracket \rho, \rho_u \rrbracket)\leq c\}}\cdot \chi(u)^{\gamma}\bigg(A^{c}_{\gamma}(u) - r_{\gamma}\Big(\frac{c}{\chi(u)}\Big)\bigg)\bigg|^q\Bigg]^{1/q}. \]

\noindent Fix $k$ and $n$. Conditioned on the types $\chi(u)$ for $u\in\mathcal{L}^k_n$, the variables $A_\gamma^c(u)$ are independent with mean $r_{\gamma}(\frac{c}{\chi(u)})$. We may then apply the Marcinkiewicz-Zygmund inequality to bound the display above by
\[
c(q)\sum^{\infty}_{k=-b}\sum^{\infty}_{n=1} \mathbb{E}\bigg[\sum_{u\in\mathcal{L}^k_n}\mathds{1}_{\{g(\llbracket \rho, \rho_u \rrbracket)\leq c\}}\cdot(\chi(u))^{\gamma q}\bigg]^{1/q} \cdot E_1\bigg[\bigg| \int^\zeta_0 \exp(\gamma\xi(t))\cdot\mathds{1}_{\{\tau^{\log(c/\chi(u))}>t\}}\mathrm{d}t \bigg|^q\bigg]^{1/q},
\]
where $c(q)$ is a positive constant depending on $q$. We first bound the second expectation. Removing the indicator, for $\gamma\in (\omega_-, \omega_- + \Delta_0)$, we have
\[
\sup_{\gamma\in[\omega_-, \omega_- + \Delta_0]}E_1\bigg[\bigg|\int^\zeta_0 \mathrm{e}^{\gamma\xi(t)}\mathrm{d}t \bigg|^q\bigg] \\ \leq  E_1\bigg[\bigg|\int^\zeta_0 \mathrm{e}^{\omega_-\xi(t)}\mathrm{d}t \bigg|^q\bigg] +  E_1\bigg[\bigg| \int^\zeta_0 \mathrm{e}^{(\omega_-+\Delta_0)\xi(t)}\mathrm{d}t \bigg|^q\bigg].
\]
\cref{Assumption B} and \cite[Lemma 9.1]{BertoinJean2024SMta}  imply that the right hand side is finite. Thus by \eqref{eq: sum L k n}, for constants $C,C'>0$,
\begin{eqnarray*}
\mathbb{E}\Big[|\tilde\lambda^{\gamma}_c(T^c)-\lambda^{\gamma}_c(T^c)|\Big]
&\leq& C c(q)\sum^{\infty}_{k=-b}\sum^{\infty}_{n=1} \mathbb{E}\bigg[ \sum_{u\in\mathcal{L}^k_n} \mathds{1}_{\{g(\llbracket\rho, \rho_u\rrbracket)\leq c\}}\cdot(\chi(u))^{\gamma q}\bigg]^{1/q}\\
&\stackrel{\eqref{eq: sum L k n}}{\leq}& Cc(q)\sum^{\infty}_{k=-b} \sum^{\infty}_{n=1} \mathrm{e}^{-(\gamma q-\omega_-)k/q}\Big(1-\frac{C_2}{k+b+2}\Big)^{(n-1)/q}
\leq C' c^{(\gamma q-\omega_-)/q}.
\end{eqnarray*}

\end{proof}

The next lemma states that the difference between $\tilde{\lambda}^{\gamma}_c(T^c)$ and its expectation conditioned on a large generation is small compared with the scales.

\begin{lem}\label{lem: approx by generation}
Uniformly in $\gamma \in (\omega_-, \omega_- + \Delta_0)$, as $n\to\infty$,
\begin{equation}\label{eq: approx by generation}
(\gamma-\omega_-)\mathbb{E}\bigg[\Big|\tilde{\lambda}^{\gamma}_c(T^c) -  \mathbb{E}[\tilde{\lambda}^{\gamma}_c(T^c)|\mathcal{F}_n]\Big|\bigg]\to 0.
\end{equation}
\end{lem}

We postpone the proof of \cref{lem: approx by generation} and show in advance that this lemma together with \cref{lem: transfer length to mass} implies \cref{lem: convergence total mass}.

\begin{proof}[Proof of \cref{lem: convergence total mass}]
Fix $n>0$. Calculating the conditional expectation, we have
\begin{equation}\label{eq: two parts length measure}
        \mathbb{E}[\tilde{\lambda}^{\gamma}_{c}(T^c)|\mathcal{F}_n]= \sum_{|u| = n}(\chi(u))^{\gamma}R_{\gamma}(c/\chi(u)) \mathds{1}_{\{g(\llbracket \rho, \rho_u \rrbracket) \leq c\}}
+ \sum_{|u| < n} (\chi(u))^{\gamma}r_{\gamma}(c/\chi(u)) \mathds{1}_{\{g(\llbracket \rho, \rho_u \rrbracket) \leq c\}}.
\end{equation}

\noindent Applying \cref{lem: many-to-one}, the expectation of the second term on the right hand side is bounded by
\[{\tt E}\bigg[\sum^{n-1}_{k=0} \mathrm{e}^{-\gamma{\hat{\tt S}}_k} r_{\gamma}(\log c + {\hat{\tt S}}_k) \bigg],\]

\noindent which remains bounded for $\gamma\in (\omega_-, \omega_- + \Delta_0)$ by \eqref{eq: rough bound r c} and the fact ${\tt E}[\exp((\gamma-\omega_-){\hat{\tt S}}_1)]< \infty$. By \cref{prop: property R gamma c}(1) which will be proved in the next section, as $\gamma\to \omega_-$,
\begin{equation}\label{eq: converge of R gamma a.s.}
    \frac{\gamma-\omega_-}{Kc_0^-} \sum_{|u| = n}(\chi(u))^{\gamma}R_{\gamma}(c/\chi(u)) \mathds{1}_{\{g(\llbracket \rho, \rho_u \rrbracket) \leq c\}} \longrightarrow D^c_n, \quad \mathbb{P}_1\mbox{-a.s.}.
\end{equation}
\noindent Again by \cref{prop: property R gamma c}(1), we have
\begin{align*}
    \frac{\gamma-\omega_-}{Kc_0^-}\mathbb{E}\left[\mathbb{E}[\tilde{\lambda}^{\gamma}_{c}(T^c)|\mathcal{F}_n]\right]=\frac{\gamma-\omega_-}{Kc_0^-} R_\gamma(c)\longrightarrow R(\log c)=\mathbb E[D^c_n].
\end{align*}
Since the expectation of the second term on the right hand side of \eqref{eq: two parts length measure} is bounded, it follows that
\begin{align}\label{eq: converge of R gamma expectation}
     \frac{\gamma-\omega_-}{Kc_0^-}\mathbb{E}\bigg[\sum_{|u| = n}(\chi(u))^{\gamma}R_{\gamma}(c/\chi(u)) \mathds{1}_{\{g(\llbracket \rho, \rho_u \rrbracket) \leq c\}}\bigg]\longrightarrow \mathbb E[D^c_n].
\end{align}
By \eqref{eq: converge of R gamma a.s.} and \eqref{eq: converge of R gamma expectation}, we apply the Scheff\'e’s lemma and conclude that
\begin{equation}\label{eq: length convergence L^1}
    \mathbb E\left[\left|\frac{\gamma-\omega_-}{Kc_0^-}\mathbb{E}[\tilde{\lambda}^{\gamma}_{c}(T^c)|\mathcal{F}_n]-D_n^c\right|\right]\to 0.
\end{equation}

The martingale $(D^c_n)_n$ converges to $\mu^c(T^c)$, as  $n\to\infty$, a.s. under $\mathbb{P}_1$ and in $L^1(\mathbb{P}_1)$. For $\varepsilon > 0$, by \cref{lem: approx by generation} we choose $n$ large enough such that
\[\frac{(\gamma-\omega_-)}{Kc_0^-}\mathbb{E}\bigg[\Big|\tilde{\lambda}^{\gamma}_c(T^c) - \mathbb{E}[\tilde{\lambda}^{\gamma}_c(T^c)| \mathcal{F}_n]\Big|\bigg]\leq \frac{\varepsilon}{4} \qquad \mbox{and} \qquad \mathbb{E}\Big[\Big|\mu^c(T^c) - D^c_n\Big|\Big]\leq \frac{\varepsilon}{4}. \]
Then when $\gamma\in(\omega_-, \omega_- + \Delta)$, we have by \cref{lem: transfer length to mass}
\[\frac{(\gamma-\omega_-)}{Kc_0^-}\mathbb{E}\Big[\Big|\tilde{\lambda}^{\gamma}_c(T^c)-\lambda^{\gamma}_c(T^c)\Big|\Big]\leq \frac{\varepsilon}{4}\]
and by \eqref{eq: length convergence L^1}
\[
\mathbb{E}\Bigg[\Bigg|\mathbb{E}\bigg[\frac{(\gamma-\omega_-)}{Kc_0^-}\tilde{\lambda}^{\gamma}_c(T^c)\bigg| \mathcal{F}_n\bigg]-D^c_n\Bigg|\Bigg]\leq \frac{\varepsilon}{4}.
\]
These four displays imply the desired convergence in $L^{1}(\mathbb{P})$.
\end{proof}

\subsubsection{More on the expected length}\label{sec: expected length}

Towards the proof of \cref{lem: approx by generation}, we begin with a technical result on the expectation of the total mass of $\lambda^{\gamma}_c$ introduced in \cref{prop: expected length}.

\begin{prop}[Properties of $R_{\gamma}(c)$]\label{prop: property R gamma c}
The following statements hold.

\begin{enumerate}
    \item For fixed $c>0$, as $\gamma\downarrow \omega_-$, we have
    \begin{equation}\label{eq: R gamma c asymptotics}
    (\gamma-\omega_-)R_{\gamma}(c)\to Kc_0^-R(\log c).
    \end{equation}

    \item For $0<\varepsilon<\Delta_0$, there exists a constant $C>0$, such that for all $c>0$ and $\gamma \in (\omega_-+\varepsilon, \omega_-+\Delta_0)$, we have
    \begin{equation}\label{eq: R gamma c bounds}
    (\gamma - \omega_-)R_{\gamma}(c) \leq C\cdot c^{\gamma-\omega_-}.
    \end{equation}

    \item There exists a constant $C>0$, such that for all $c>0$ and $\gamma \in (\omega_-, \omega_-+\Delta_0)$, we have
    \begin{equation}\label{eq: R gamma c bounds 2}
    (\gamma - \omega_-)R_{\gamma}(c) \leq C\cdot c^{(\gamma-\omega_-)}(1+\log_+ c).
    \end{equation}
    \item There exists a constant $C>0$ such that for $c_1, c_2>0$ and $\gamma \in (\omega_-, \omega_-+\Delta_0)$, we have
    \begin{equation}\label{eq: R gamma c fraction bounds}
    \frac{R_{\gamma}(c_1)}{R_{\gamma}(c_2)}\leq C \Big(\frac{c_1}{c_2}\Big)^{(\gamma-\omega_-)}(1+\log_+(c_1/c_2)).
    \end{equation}
    Here $\log_+(x):=\max(\log x, 0)$.
\end{enumerate}

\end{prop}

\begin{proof}
We first prove \eqref{eq: R gamma c asymptotics}. As $\gamma\downarrow\omega_-$, we have
$\psi(\gamma)\to \psi(\omega_-)$ and
\[\int_{[0, \log c]} \mathrm{e}^{(\gamma-\omega_-)x}R(\mathrm{d}x) \to R([0, \log c]) = R(\log c)\]
\noindent by monotone convergence theorem. By integration by parts, we have
\begin{align*}
    (\gamma-\omega_-)\int_{[0, \infty)} \mathrm{e}^{-(\gamma-\omega_-)x}R^+(\mathrm{d}x) &= (\gamma-\omega_-)^2 \int^{\infty}_0 \mathrm{e}^{-(\gamma-\omega_-)x}R^+(x)\mathrm{d}x \\
&= (\gamma-\omega_-) \int^{\infty}_0 \mathrm{e}^{-y}R^+(\frac{y}{\gamma-\omega_-})\mathrm{d}y.
\end{align*}
\noindent Here we change variables by setting $y = (\gamma-\omega_-)x$. Since $(\gamma-\omega_-)R^+(y/(\gamma-\omega_-))$ converges to $c_0^-$ and is bounded by a linear function in $y$, the dominated convergence theorem implies that the last display converges to $c_0^-$.

To prove \eqref{eq: R gamma c bounds} and \eqref{eq: R gamma c bounds 2}, notice that the term $(\gamma-\omega_-)\int_{[0,\infty)} \mathrm{e}^{-(\gamma-\omega_-)x}R^+(\mathrm{d}x)$ is bounded. It remains to control the integral $\int_{[0, \log c]}\mathrm{e}^{(\gamma-\omega_-)x}R(\D x)$. From \eqref{eq: renewal theorem}, there exist  constants $M>a\ge 1$ and $c_+\ge c_->0$ such that for $x\ge M$, $ c_-a \le  R([x,x+a])\le c_+a$. For each $n\geq 1$, we have
\begin{align*}
    K(n) := \int_{[M+a(n-1), M+an]}\mathrm{e}^{(\gamma-\omega_-)x} R(\D x) &\leq \mathrm{e}^{(\gamma-\omega_-)(M+an)}c_+a\\
   & \leq \mathrm{e}^{(\gamma-\omega_-)a}c_+\int^{M+an}_{M+a(n-1)}\mathrm{e}^{(\gamma-\omega_-)x}\D x
\end{align*}

\noindent and
\[K(n)\geq \mathrm{e}^{(\gamma-\omega_-)(M+(a-1)n)}c_-a\geq \mathrm{e}^{-(\gamma-\omega_-)a}c_-\int^{M+an}_{M+a(n-1)}\mathrm{e}^{(\gamma-\omega_-)x}\D x.\]

\noindent When $\log c\le 2M$, we see that $R_{\gamma}(0)\leq R_{\gamma}(c_i)\leq R_{\gamma}(2M)$. When $\log c>2M$, we have the upper bound
\begin{equation}\label{eq: R gamma upper bound}
R_{\gamma}(c)\leq R_{\gamma}(2M)+c_+\sum_{n=1}^{[(\log c-M)/a]+1}K(n)
\le  R_{\gamma}(M)+ \mathrm{e}^{(\gamma-\omega_-)a}c_+\int^{\log c+a}_{M} \mathrm{e}^{(\gamma-\omega_-)x}\D x
\end{equation}
\noindent and the lower bound
\begin{equation}\label{eq: R gamma lower bound}
    R_{\gamma}(c)\ge R_{\gamma}(M)+c_-\sum_{n=1}^{[(\log c-M)/a]}K(n)\ge R_{\gamma}(M)+\mathrm{e}^{-(\gamma-\omega_-)a}c_-\int^{\log c-a}_{M} \mathrm{e}^{(\gamma-\omega_-)x}\D x.
\end{equation}

\noindent By calculating the integral $ \int^{\log c-a}_{M} \mathrm{e}^{(\gamma-\omega_-)x}\D x$, \eqref{eq: R gamma c bounds} and \eqref{eq: R gamma c bounds 2} are direct consequences of \eqref{eq: R gamma upper bound}.

We finally prove \eqref{eq: R gamma c fraction bounds}. By \eqref{eq: expected length}, we have
\[\frac{R_{\gamma}(c_1)}{R_{\gamma}(c_2)} = \frac{\int_{[0, \log c_1]} \mathrm{e}^{(\gamma-\omega_-)x}R(\mathrm{d}x)}{\int_{[0, \log c_2]} \mathrm{e}^{(\gamma-\omega_-)x}R(\mathrm{d}x)}. \]

\noindent We argue by distinguishing cases depending on the values of $c_1$ and $c_2$. When $\log c_1, \log c_2\leq 2M$, the fraction is bounded by constants. When $\log c_1 > 2M\geq \log c_2$ ($\log c_2 > 2M\geq \log c_1$), $R_{\gamma}(c_2)$ ($R_{\gamma}(c_1)$) is bounded from below (above), and the upper (lower) bound in \eqref{eq: R gamma upper bound} (\eqref{eq: R gamma lower bound}) gives the desired bound. When $\log c_1, \log c_2>2M$, there exists a constant $C>0$ such that
\[
\frac{R_{\gamma}(c_1)}{R_{\gamma}(c_2)}\leq \frac{R_{\gamma}(M)+\mathrm{e}^{(\gamma-\omega_-)a}c_+\int^{\log c_1+a}_{M} \mathrm{e}^{(\gamma-\omega_-)x}\D x}{R_{\gamma}(M)+\mathrm{e}^{-(\gamma-\omega_-)a}c_-\int^{\log c_2-a}_{M} \mathrm{e}^{(\gamma-\omega_-)x}\D x}
\le C\Big(\frac{c_1}{c_2}\Big)^{\gamma-\omega_-} \left(1+\log_+(c_1/c_2)\right).
\]

\noindent We conclude by combining all four cases.
\end{proof}

For future use, we put a random walk analogue of the estimates in \cref{prop: property R gamma c}.

\begin{Cor}
For the random walk $({\hat{\tt S}}_n)_{n\geq 0}$, we have the following discrete analogue of \cref{prop: property R gamma c}(2): For $0<\varepsilon<\Delta_0$, there exists a constant $C>0$ such that for all $b>0$ and $\gamma \in (\omega_-+\varepsilon, \omega_-+\Delta_0)$, we have
\begin{equation}\label{eq: green function r.w.}
(\gamma-\omega_-)\cdot{\tt E}\left[\sum^{{\tt T}^+_b-1}_{n=0} \mathrm{e}^{-(\gamma-\omega_-){\hat{\tt S}}_n}\right]\leq C\mathrm{e}^{(\gamma-\omega_-)b}.
\end{equation}
\end{Cor}

\subsubsection{Proof of \cref{lem: approx by generation}: good and bad branches}
To prove \cref{lem: approx by generation}, we use a similar idea to that in \cite[Section 5]{AidekonElie2026TSLo} by introducing a classification of decorated branches into good and bad branches. The criterion is chosen such that the bad branches is small in $L^1(\mathbb{P})$, while the good branches is bounded in $L^2(\mathbb{P})$. We consider the following condition analogous to \cite[Definition 5.4]{AidekonElie2026TSLo}. The heuristic is to control the increment of the expected length above each generation along the lineage of the vertex. It could be viewed as
a decaying condition compared to the estimates of $R_{\gamma}(c)$ in \cref{prop: property R gamma c}.
\begin{defn}\label{def: good points}
Set $\theta = \frac{\omega_-}{4}$. Let $(c, A)$ be a pair of parameters with $c>1$ and $A>1$. We say that a decorated branch with label $u\in\mathbb{U}^c$ is \textit{good} if for $0\leq k\leq |u|-1$,
\begin{equation}\label{eq: def good points}
\sum^{\infty}_{i=1} \bigg(\frac{\chi(u_ki)}{\chi(u_k)}\bigg)^{\omega_-}\Big(1+\ln_+\Big(\frac{\chi(u_k)}{\chi(u_ki)}\Big)\Big)\leq A\bigg(\frac{c}{\chi(u_k)}\bigg)^{\theta}.
\end{equation}

\noindent We say that a decorated branch is \textit{bad} if it is not good. Denote by $\mathbb{B}^c$ and $\mathbb{G}^c$ the set of labels of bad and good regions, and by $T^c_{G}$ and $T^c_{B}$ for the subset of $T^c$ consisting of good and bad decorated branches, respectively.
\end{defn}

\noindent With the classification, we could state our estimates on the good and bad branches.

\begin{lem}\label{lem: estimate bad}
There exists a constant $C>0$, such that for $(c, A)$ and $\gamma\in(\omega_-, \omega_- + \Delta_0)$, we have
\begin{equation}\label{eq: estimate bad}
\mathbb{E}[(\gamma-\omega_-)\tilde{\lambda}^{\gamma}_c(T^c_B)]\leq CA^{-\eta} c^{(\gamma-\omega_-)}.
\end{equation}
\end{lem}

\begin{lem}\label{lem: estimate good}
There exists a constant $C>0$, such that for $(c, A)$ and $\gamma\in(\omega_-, \omega_- + \Delta_0)$, we have
\begin{equation}
\mathbb{E}[(\gamma-\omega_-)^2\tilde{\lambda}^{\gamma}_c(T^c_G)^2]\leq CA^2c^{(2\gamma-\omega_-)}.
\end{equation}
\end{lem}

Let us see how these two lemmas imply \cref{lem: approx by generation}.

\begin{proof}[Proof of \cref{lem: approx by generation}]
Fix $\varepsilon>0$. By \cref{lem: estimate bad}, we choose $A>0$ such that uniformly for $\gamma\in(\omega_-, \omega_-+\Delta_0)$,
\[(\gamma-\omega_-)\mathbb{E}\Big[\Big|\tilde{\lambda}^{\gamma}_c(T^c_B) - \mathbb{E}\big[\tilde{\lambda}^{\gamma}_c(T^c_B)\big| \mathcal{F}_n\big]\Big|\Big]\leq 2\mathbb{E}[(\gamma-\omega_-)\tilde{\lambda}^{\gamma}_c(T^c_B)] \leq \frac{\varepsilon}{3}. \]
\noindent It suffices to show that when $n$ is large enough, for $\gamma \in (\omega_-, \omega_-+\Delta_0)$,
\begin{equation}\label{eq: pass to generatio n}
(\gamma-\omega_-)\mathbb{E}\Big[\Big|\tilde{\lambda}^{\gamma}_c(T^c_G(n)) - \mathbb{E}\big[\tilde{\lambda}^{\gamma}_c(T^c_G(n))\big| \mathcal{F}_n\big]\Big| \mathds{1}_{\big\{(\gamma-\omega_-)\big|\tilde{\lambda}^{\gamma}_c(T^c_G) - \mathbb{E}[\tilde{\lambda}^{\gamma}_c(T^c_G)| \mathcal{F}_n]\big|>\frac{\varepsilon}{3}\big\}}\Big] \leq \frac{\varepsilon}{3}.
\end{equation}

\noindent The left-hand side is smaller than
\begin{multline*}
\frac{3}{\varepsilon}(\gamma-\omega_-)^2\mathbb{E}\Big[\Big|\tilde{\lambda}^{\gamma}_c(T^c_G(n)) - \mathbb{E}\big[\tilde{\lambda}^{\gamma}_c(T^c_G(n))\big|\mathcal{F}_n\big]\Big|^2\Big]
\leq \frac{3}{\varepsilon}\mathbb{E}\bigg[\sum_{|u|=n}\mathds{1}_{\{u\in\mathbb{G}^c\}}(\gamma-\omega_-)^2(\tilde{\lambda}^{\gamma}_c(T^c_G(u)))^2\bigg].
\end{multline*}

\noindent Here, $T^c_G(u)$ is the subtree consisting of good branches rooted at $u$. By Markov property and the scaling property, the law of $\tilde{\lambda}^{\gamma}_c(T^c_G(u))$ condition on $\mathcal{F}_n$ has the same law as $(\chi(u))^{\gamma}\tilde{\lambda}^{\gamma}_{c'}(T^{c'}_G)$ where $c' = c/\chi(u)$ and the good branches are defined according to the parameters $(c', A)$. We apply \cref{lem: estimate good} to bound the display above by
\begin{eqnarray*}
\frac{3CA^2}{\varepsilon}\mathbb{E}\Bigg[\sum_{|u|=n}\mathds{1}_{\{u\in\mathbb{G}^c\}} (\chi(u))^{2\gamma}\Big(\frac{c}{\chi(u)}\Big)^{2\gamma-\omega_-}\Bigg]
\leq\frac{3CA^2c^{2\gamma-\omega_-}}{\varepsilon}\mathbb{E}\left[\sum_{|u|=n}\mathds{1}_{\{\chi(u_k)\leq c, k\leq n\}} (\chi(u))^{\omega_-}\right].
\end{eqnarray*}

\noindent Applying the many-to-one formula \eqref{eq: many-to-one}, we have
\[ \mathbb{E}\Bigg[\sum_{|u|=n}\mathds{1}_{\{\chi(u_k)\leq c, k\leq n\}} (\chi(u))^{\omega_-}\Bigg]={\tt P}({\hat{\tt S}}_k\geq -\log c, 1\leq k\leq n)\leq \frac{C{\tt R}(\log c)}{\sqrt{n}}. \]

\noindent The last inequality is from \cite[Equation (A.7)]{ShiZhan2016BRWE}, where the constant $C>0$ only depends on the random walk. Now \eqref{eq: pass to generatio n} holds uniformly for $\gamma\in(\omega_-, \omega_-+\Delta_0)$ when $n$ is large.

\end{proof}

Now we return to the proofs of \cref{lem: estimate bad} and \cref{lem: estimate good}.
\begin{proof}[Proof of \cref{lem: estimate bad}]
Throughout the proof we allow the universal constant $C>0$ to vary from line to line. For simplicity let $g(x): = x^{\omega_-}(1+\ln_+(1/x))$ and let $q=p-1$ with $p$ in \eqref{eq: uniform p moment}. For $0<\delta<\Delta_0\wedge\theta q$, there exists $C = C(\delta)>0$ such that $\ln_+(1/x)\leq Cx^{-\delta}$. Then $g(x)\leq x^{\omega_-}+Cx^{\omega_--\delta}$. Using Minkowski's inequality and \eqref{eq: uniform p moment}, we have
\begin{equation}\label{eq: sum g moment}
\mathbb{E}\bigg[\bigg(\sum^{\infty}_{i=1} g(\chi(i))\bigg)^{1+ q}\bigg] \leq C\mathbb{E} \bigg[\bigg(\sum^{\infty}_{i=1}(\chi(i)^{\omega_-}+\chi(i)^{\omega_-\delta})\bigg)^{1+ q}\bigg] < \infty.
\end{equation}

Let $\mathbb{B}^c_{\leftarrow}$ be the set of labels $v\in\mathbb{U}$ such that $v\in\mathbb{G}^c$, but \eqref{eq: def good points} fails with $u_k$ replaced by $v$. Conversely, if $u\in\mathbb{B}^c$, consider the smallest $k$ such that \eqref{eq: def good points} fails. Then $u_k\in\mathbb{B}^c_{\leftarrow}$. In conclusion, $\mathbb{B}^c = \{vu: v\in \mathbb{B}^c_{\leftarrow}, u\in\mathbb{U}^* \}$. We could rewrite the expected total length of the bad branches by
\[\mathbb{E}[\tilde{\lambda}^{\gamma}_c(T^c_B)] = \sum^{\infty}_{n=0}\mathbb{E}\Bigg[\sum_{|v|=n}\mathds{1}_{\{v\in\mathbb{G}^c\}} \mathds{1}_{\big\{\sum^{\infty}_{i=1}g(\frac{\chi(vi)}{\chi(v)}) > A(\frac{c}{\chi(v)})^{\theta}\big\}}\sum_{u\in\mathbb{U}^*}(\chi(vu))^{\gamma} r_{\gamma}\Big(\frac{c}{\chi(vu)}\Big) \mathds{1}_{\{vu\in\mathbb{U}^c\}}\Bigg]. \]

\noindent For simplicity we introduce the event $B_x = \{\sum^{\infty}_{i=1}g(\frac{\chi(i)}{\chi(\varnothing)})>A(\frac{c}{x})^{\theta}\}$. For fixed $n\geq 0$, we take conditional expectation on $\mathcal{F}_{n}$. The last display is equal to
\begin{equation}\label{eq: proof bad region 1}
\sum^{\infty}_{n=0}\mathbb{E} \Bigg[\sum_{|v|=n}\mathds{1}_{\{v\in\mathbb{G}^c\}} \mathbb{E}_x\bigg[\mathds{1}_{B_x} \sum_{u\in\mathbb{U}^*} (\chi(u))^{\gamma} r_{\gamma}\Big(\frac{c}{\chi(u)}\Big)\mathds{1}_{\{u\in\mathbb{U}^c\}}  \bigg]\Bigg|_{x=\chi(v)}\Bigg].
\end{equation}

\noindent Let $c(x) = c/x$. By the scaling property,
\begin{align}
\mathbb{E}_x\bigg[\mathds{1}_{B_x} \sum_{u\in\mathbb{U}^*} (\chi(u))^{\gamma}r_{\gamma}\Big(\frac{c}{\chi(u)}\Big) \mathds{1}_{\{u\in\mathbb{U}^c\}} \bigg]&= x^{\gamma}\mathbb{E}\bigg[\mathds{1}_{B_x} \sum_{u\in\mathbb{U}^*} (\chi(u))^{\gamma}r_{\gamma}\Big(\frac{c(x)}{\chi(u)}\Big)\mathds{1}_{\{u\in\mathbb{U}^{c(x)}\}} \bigg]\nonumber\\
&= x^{\gamma}\mathbb{E}\bigg[\mathds{1}_{B_x} \sum^{\infty}_{i=1} \tilde{\lambda}^{\gamma}_{c(x)}(T^{c(x)}(i))\bigg]. \label{eq: proof bad 1}
\end{align}

\noindent Taking conditional expectation with respect to $\mathcal{F}_1$, by \eqref{eq: expected length 2} we get
\[\mathbb{E}\bigg[\mathds{1}_{B_x} \sum^{\infty}_{i=1} \tilde{\lambda}^{\gamma}_{c(x)}(T^{c(x)}(i))\bigg] = \mathbb{E}\bigg[\mathds{1}_{B_x} \sum^{\infty}_{i=1} (\chi(i))^{\gamma}R_{\gamma}\Big(\frac{c(x)}{\chi(i)}\Big)\bigg].\]

\noindent Using the bounds of the fraction in \eqref{eq: R gamma c fraction bounds}(iv), we have
\begin{equation}\label{eq: frac bound by g}
\frac{(\chi(i))^{\gamma}R_{\gamma}\Big(\frac{c(x)}{\chi(i)}\Big)}{R_{\gamma}(c(x))} \leq C(\chi(i))^{\gamma}\Big(1 + \log_+\Big(\frac{1}{\chi(i)}\Big)\Big) = Cg(\chi(i)).
\end{equation}

\noindent Hence, it follows by the definition of the set $B_x$ that
\begin{align}
\mathbb{E}\bigg[\mathds{1}_{B_x} \sum^{\infty}_{i=1} (\chi(i))^{\gamma}R_{\gamma}\Big(\frac{c(x)}{\chi(i)}\Big)\bigg]&\leq CR_{\gamma}(c(x))\mathbb{E}\bigg[\mathds{1}_{B_x}\sum^{\infty}_{i=1}g(\chi(i))\bigg]\nonumber\\
&\leq CR_{\gamma}(c(x))A^{- q}c^{-\theta q}x^{\theta q}\mathbb{E}\bigg[\bigg(\sum^{\infty}_{i=1} g(\chi(i))\bigg)^{1+ q}\bigg]. \label{eq: proof bad 2}
\end{align}

\noindent Combining \eqref{eq: sum g moment}, \eqref{eq: proof bad 1} and \eqref{eq: proof bad 2}, we conclude that
\[\mathbb{E}[\tilde{\lambda}^{\gamma}_c(T^c_B)]\leq CA^{- q}c^{-\theta q}\sum^{\infty}_{n=0}\mathbb{E} \Bigg[\sum_{|u|=n}(\chi(u))^{\gamma+\theta q} R_{\gamma}\Big(\frac{c}{\chi(u)}\Big)\mathds{1}_{\{u\in\mathbb{G}^c\}}\Bigg]. \]

\noindent Replacing the event $\{u\in\mathbb{G}^c\}$ by $\{u\in \mathbb U^c\}$ results in an upper bound. Applying \cref{lem: many-to-one},
\[\mathbb{E}[\tilde{\lambda}^{\gamma}_c(T^c_B)]\leq CA^{- q}c^{-\theta q}{\tt E}\bigg[\sum^{{\tt T}_{\log c}^+-1}_{n=0} \mathrm{e}^{-(\gamma-\omega_-+\theta q){\hat{\tt S}}_n}R_{\gamma}(c\mathrm{e}^{\hat{\tt S}_n})\bigg].\]

\noindent It entails from \eqref{eq: R gamma c bounds} that
\begin{equation}\label{eq: R gamma bounds 2}
(\gamma-\omega_-)R_{\gamma}(c\mathrm{e}^{{\hat{\tt S}}_n})\leq Cc^{\gamma-\omega_-}\mathrm{e}^{(\gamma-\omega_-){\hat{\tt S}}_n}.
\end{equation}

\noindent Therefore,
\[\mathbb{E}[(\gamma-\omega_-)\tilde{\lambda}^{\gamma}_c(T^c_B)]\leq CA^{- q}c^{\gamma-\omega_-}c^{-\theta q}{\tt E}\bigg[\sum^{{\tt T}_{\log c}^+-1}_{n=0} \mathrm{e}^{-\theta q{\hat{\tt S}}_n}\bigg]\leq CA^{- q}c^{\gamma-\omega_-}, \]

\noindent where in the last inequality we use \eqref{eq: green function r.w.}.
\end{proof}

\begin{proof}[Proof of \cref{lem: estimate good}]
Throughout the proof we allow the universal constant $C>0$ to vary from line to line. We use double indices $u$ and $v$ to interpret $\tilde{\lambda}^{\gamma}_c(T^c_G)^2$ by
\[\tilde{\lambda}^{\gamma}_c(T^c_G)^2 = \bigg(\sum_{u\in\mathbb{G}^c}(\chi(u))^{\gamma}r_{\gamma}\Big(\frac{c}{\chi(u)}\Big)\bigg)\bigg(\sum_{v\in\mathbb{G}^c}(\chi(v))^{\gamma}r_{\gamma}\Big(\frac{c}{\chi(v)}\Big)\bigg). \]

\noindent When expanding the brackets, there will be three cases for $(u, v)$:
\begin{enumerate}
   \item $u = v$;
   \item $u\ne v$, either $u\prec v$ or $v\prec u$ is satisfied;
   \item $u\ne v$, neither $u\prec v$ nor $v\prec u$ is satisfied.
\end{enumerate}
\noindent In formulation, we write
\begin{align*}
\tilde{\lambda}^{\gamma}_c(T^c_G)^2 =& \sum_{u\in\mathbb{G}^c}(\chi(u))^{2\gamma}r_{\gamma}\Big(\frac{c}{\chi(u)}\Big)^2 + 2\sum_{u\in\mathbb{G}^c} \sum_{v\in\mathbb{G}^c} \mathds{1}_{\{u\ne v\}}\mathds{1}_{\{u\prec v\}}(\chi(u))^{\gamma}r_{\gamma}\Big(\frac{c}{\chi(u)}\Big)(\chi(v))^{\gamma}r_{\gamma}\Big(\frac{c}{\chi(v)}\Big)\\
& + \sum_{u\in\mathbb{G}^c} \sum_{v\in\mathbb{G}^c} \mathds{1}_{\{u\ne v\}}\mathds{1}_{\{u\nprec v\}}\mathds{1}_{\{v\nprec u\}} (\chi(u))^{\gamma}r_{\gamma}\Big(\frac{c}{\chi(u)}\Big)(\chi(v))^{\gamma}r_{\gamma}\Big(\frac{c}{\chi(v)}\Big)\\
=& I_1 + I_2 + I_3.
\end{align*}

We first estimate $\mathbb{E}[I_1]$. Using the rough bound $r_{\gamma}(c/\chi(u))\le |\psi(\gamma)|^{-1}$ by \eqref{eq: rough bound r c}, and replacing $\{u\in\mathbb{G}^c\}$ by $\{\chi(u_k)\leq c, k\leq |u|\}$, we see
\[
\mathbb{E}[I_1] \leq |\psi(\gamma)|^{-2}\mathbb{E}\bigg[\sum_{u\in\mathbb{U}}(\chi(u))^{2\gamma}\mathds{1}_{\{\chi(u_k)<c, k\leq |u|\}}\bigg].
\]
By the many-to-one formula \eqref{eq: many-to-one} and then \eqref{eq: green function r.w.}, we have
\begin{equation}\label{eq: proof good I1 final}
\mathbb{E}[I_1] \leq |\psi(\gamma)|^{-2}{\tt E}\bigg[\sum^{{\tt T}_{\log c}^+-1}_{n=0} \mathrm{e}^{-(2\gamma-\omega_-){\hat{\tt S}}_n}\bigg]\leq Cc^{2\gamma-\omega_-}.
\end{equation}

\noindent Notice that $2\gamma-\omega_-$ is bounded away from $0$, so the last inequality is uniform in $\gamma\in(\omega_-, \omega_-+\Delta_0)$.

We continue to estimate $\mathbb{E}[I_2]$. For $u\prec v$ and $u\ne v$, write $v = uw$ for $w\in \mathbb{U}^*$. Rewrite the summation by
\[
\mathbb{E}[I_2] = 2\sum^{\infty}_{n=0}\mathbb{E}\bigg[\sum_{|u|=n}\mathds{1}_{\{u\in\mathbb{G}^c\}} (\chi(u))^{\gamma}r_{\gamma}\Big(\frac{c}{\chi(u)}\Big)\sum_{w\in\mathbb{U}^*} (\chi(uw))^{\gamma}r_{\gamma}\Big(\frac{c}{\chi(uw)}\Big)\mathds{1}_{\{uw\in\mathbb{G}^c\}}\bigg].
\]

\noindent Define the event (the complement of $B_x$)
\begin{equation}\label{eq: def G_x}
    G_x = \bigg\{\sum^{\infty}_{i=1} g\Big(\frac{\chi(i)}{\chi(\varnothing)}\Big)\leq A\Big(\frac{c}{x}\Big)^{\theta} \bigg\}.
\end{equation}

\noindent For each $n\geq 0$, we take conditional expectation with respect to $\mathcal{F}_n$, whence
\begin{equation}\label{eq: proof good I2 1}
\mathbb{E}[I_2] \le 2\sum^{\infty}_{n=0}\mathbb{E}\bigg[\sum_{|u|=n}\mathds{1}_{\{u\in\mathbb{U}^c\}} (\chi(u))^{\gamma}r_{\gamma}\Big(\frac{c}{\chi(u)}\Big)\mathbb{E}_x\Big[\mathds{1}_{G_x}\sum_{w\in\mathbb{U}^*} (\chi(w))^{\gamma}r_{\gamma}\Big(\frac{c}{\chi(w)}\Big)\mathds{1}_{\{w\in\mathbb{G}^{c}\}}\Big]\bigg|_{x=\chi(u)}\bigg].
\end{equation}

\noindent Recall $c(x)=c/x$. Replacing the indicator $\mathds{1}_{\{w\in\mathbb{G}^c\}}$ by $\mathds{1}_{\{w\in\mathbb{U}^c\}}$, by the scaling property, we bound the inner expectation by
\begin{align}  \mathbb{E}_x\bigg[\mathds{1}_{G_x}\sum_{w\in\mathbb{U}^*} (\chi(w))^{\gamma}r_{\gamma}\Big(\frac{c}{\chi(w)}\Big) \mathds{1}_{\{w\in\mathbb{U}^c\}}\bigg] &=x^{\gamma}\mathbb{E}\bigg[\mathds{1}_{G_x}\sum_{w\in\mathbb{U}^*} (\chi(w))^{\gamma}r_{\gamma}\Big(\frac{c(x)}{\chi(w)}\Big)\mathds{1}_{\{w\in\mathbb{U}^{c(x)}\}}\bigg]\nonumber \\
&=x^{\gamma}\mathbb{E}\bigg[\mathds{1}_{G_x}\sum^{\infty}_{i=1} \tilde{\lambda}^{\gamma}_{c(x)}(T^{c(x)}(i))\bigg]. \label{eq: proof good I2 2}
\end{align}

\noindent Take conditional expectation on $\mathcal{F}_1$,  we have
\[\mathbb{E}\bigg[\mathds{1}_{G_x}\sum^{\infty}_{i=1} \tilde{\lambda}^{\gamma}_{c(x)}(T^{c(x)}(i))\bigg] = \mathbb{E}\bigg[\mathds{1}_{G_x}\sum^{\infty}_{i=1} (\chi(i))^{\gamma}R_{\gamma}\Big(\frac{c(x)}{\chi(i)}\Big)\bigg]. \]

\noindent With the inequality \eqref{eq: frac bound by g} and the definition of $G_x$, it follows that
\[\mathbb{E}\bigg[\mathds{1}_{G_x}\sum^{\infty}_{i=1} (\chi(i))^{\gamma}R_{\gamma}\Big(\frac{c(x)}{\chi(i)}\Big)\bigg]\leq CR_{\gamma}(c(x))\mathbb{E}\bigg[\mathds{1}_{G_x}\sum^{\infty}_{i=1} g(\chi(i))\bigg]\leq CAc^{\theta}x^{-\theta}R_{\gamma}(c(x)). \]

\noindent Combining with \eqref{eq: proof good I2 2}, we bound \eqref{eq: proof good I2 1} by
\[\mathbb{E}[I_2]\leq CAc^{\theta}\sum^{\infty}_{n=0}\mathbb{E}\bigg[\sum_{|u|=n}\mathds{1}_{\{u\in\mathbb{G}^c\}} (\chi(u))^{2\gamma-\theta}R_{\gamma}\Big(\frac{c}{\chi(u)}\Big)\bigg].\]

\noindent Applying \cref{lem: many-to-one}, we rewrite
\[
\mathbb{E}[I_2]\leq CAc^{\theta}{\tt E}\bigg[\sum^{{\tt T}^+_{\log c}-1}_{n=0}\mathrm{e}^{-(2\gamma-\theta-\omega_-){\hat{\tt S}}_n} R_{\gamma}(c\mathrm{e}^{{\hat{\tt S}}_n})\bigg].
\]

\noindent The bound of $R_{\gamma}(c)$ in \eqref{eq: R gamma bounds 2} now implies that
\begin{equation}\label{eq: proof good I2 final}
\mathbb{E}[(\gamma-\omega_-)I_2]\leq CAc^{\theta+\gamma-\omega_-} {\tt E}\bigg[\sum^{{\tt T}_{\log c}^+-1}_{n=0}\mathrm{e}^{-(\gamma-\theta){\hat{\tt S}}_n} \bigg]\leq CAc^{2\gamma-\omega_-}.
\end{equation}

\noindent We use \eqref{eq: green function r.w.} in the last inequality. Since $\gamma-\theta$ is bounded away from $0$, the last step is uniform in $\gamma\in(\omega_-, \omega_-+\Delta_0)$.

We finally estimate $\mathbb{E}[I_3]$. Consider the last common ancestor (denoted by $w$) of $u$ and $v$. Write $u = wu'$ and $v = wv'$. Then $u'_1\ne v'_1\in\mathbb{U}^*$. By definition, $w\in\mathbb{G}^c$. We rewrite $\mathbb{E}[I_3]$ as
\[
  \mathbb{E}\bigg[\sum_{w\in\mathbb{G}^c}\sum_{u', v'\in\mathbb{U}^*}\mathds{1}_{\{u'_1\ne v'_1\}}(\chi(wu'))^{\gamma}r_{\gamma}\Big(\frac{c}{\chi(wu')}\Big)(\chi(wv'))^{\gamma}r_{\gamma}\Big(\frac{c}{\chi(wv')}\Big)\mathds{1}_{\{wu'\in\mathbb{G}^c\}}\mathds{1}_{\{wv'\in\mathbb{G}^c\}}\bigg].
\]

\noindent We use the same argument as in the estimate of $\mathbb{E}[I_2]$. For each $n\geq 0$, taking conditional expectation with respect to $\mathcal{F}_n$ and inserting the indicator of $G_{x}$, we see that $\mathbb{E}[I_3]$ is less than
\begin{equation}\label{eq: proof good I3 1}
 \mathbb{E}\bigg[\sum_{w\in\mathbb{U}^c}\mathbb{E}_x\bigg[\mathds{1}_{G_x}\sum_{u, v\in\mathbb{U}^*}\mathds{1}_{\{u_1\ne v_1\}}(\chi(u))^{\gamma}r_{\gamma}\Big(\frac{c}{\chi(u)}\Big)(\chi(v))^{\gamma}r_{\gamma}\Big(\frac{c}{\chi(v)}\Big)\mathds{1}_{\{u\in\mathbb{G}^c\}}\mathds{1}_{\{v\in\mathbb{G}^c\}}\bigg]\bigg|_{x=\chi(w)}\bigg].
\end{equation}

\noindent Replace the indicator $\mathds{1}_{\{u,v\in\mathbb{G}^c\}}$ by $\mathds{1}_{\{u,v\in\mathbb{U}^c\}}$. By the scaling property, we bound the inner expectation by
\begin{align}
&\mathbb{E}_x\bigg[\sum_{u, v\in\mathbb{U}^*}\mathds{1}_{G_x}\mathds{1}_{\{u_1\ne v_1\}}(\chi(u))^{\gamma}r_{\gamma}\Big(\frac{c}{\chi(u)}\Big)(\chi(v))^{\gamma}r_{\gamma}\Big(\frac{c}{\chi(v)}\Big)\mathds{1}_{\{u\in\mathbb{U}^c\}}\mathds{1}_{\{v\in\mathbb{U}^c\}}\bigg]\nonumber\\
&= x^{2\gamma}\mathbb{E}\bigg[\sum_{u, v\in\mathbb{U}^*}\mathds{1}_{G_x}\mathds{1}_{\{u_1\ne v_1\}}(\chi(u))^{\gamma}r_{\gamma}\Big(\frac{c(x)}{\chi(u)}\Big)(\chi(v))^{\gamma}r_{\gamma}\Big(\frac{c(x)}{\chi(v)}\Big)\mathds{1}_{\{u\in\mathbb{U}^{c(x)}\}}\mathds{1}_{\{v\in\mathbb{U}^{c(x)}\}}\bigg]\nonumber\\
&= x^{2\gamma}\mathbb{E}\bigg[\mathds{1}_{G_x}\sum_{i\ne j}\tilde{\lambda}^{\gamma}_{c(x)}(T^{c(x)}_i)\cdot \tilde{\lambda}^{\gamma}_{c(x)}(T^{c(x)}_j)\bigg].\label{eq: proof good I3 2}
\end{align}

\noindent Condition on the first generation, $T^{c(x)}_i$ and $T^{c(x)}_j$ are independent for $i\ne j$. Taking conditional expectation, we see
\begin{align*}
\mathbb{E}\bigg[\mathds{1}_{G_x}\sum_{i\ne j}\tilde{\lambda}^{\gamma}_{c(x)}(T^{c(x)}(i))\cdot \tilde{\lambda}^{\gamma}_{c(x)}(T^{c(x)}(j))\bigg]= \mathbb{E}\bigg[\mathds{1}_{G_x}\sum_{i\ne j}(\chi(i))^{\gamma}R_{\gamma}\Big(\frac{c(x)}{\chi(i)}\Big)(\chi(j))^{\gamma}R_{\gamma}\Big(\frac{c(x)}{\chi(j)}\Big)\bigg].
\end{align*}

\noindent Using the inequality \eqref{eq: frac bound by g} and the definition of $G_x$ \eqref{eq: def G_x} again, the above display is bounded by
\[CR_{\gamma}(c(x))^2\mathbb{E} \bigg[\mathds{1}_{G_x}\sum_{i\ne j} g(\chi(i)) g(\chi(j))\bigg]\leq C A^2 R_{\gamma}(c(x))^2 x^{-2\theta}c^{2\theta}. \]

\noindent Plugging it back to \eqref{eq: proof good I3 2} then back to \eqref{eq: proof good I3 1}, we bound $\mathbb{E}[I_3]$ by
\[\mathbb{E}[I_3]\leq C A^2 c^{2\theta}\sum^{\infty}_{n=0} \mathbb{E}\bigg[\sum_{|w| = n}\mathds{1}_{\{w\in\mathbb{U}^c\}}(\chi(w))^{2\gamma-2\theta}R_{\gamma}\Big(\frac{c}{\chi(w)}\Big)^2\bigg]. \]

\noindent The many-to-one formula (\cref{lem: many-to-one}) implies that
\[\mathbb{E}[I_3]\leq C A^2 c^{2\theta}{\tt E}\bigg[\sum^{{\tt T}_{\log c}^+-1}_{n=0}\mathrm{e}^{-(2\gamma-2\theta-\omega_-){\hat{\tt S}}_n} R_{\gamma}(c\mathrm{e}^{{\hat{\tt S}}_n})^2\bigg].\]

\noindent Using the bounds \eqref{eq: R gamma bounds 2} and \eqref{eq: green function r.w.}, we have
\begin{equation}\label{eq: proof good I3 final}
\mathbb{E}[(\gamma-\omega_-)^2I_3]\leq CA^2c^{2(\theta+\gamma-\omega_-)}{\tt E}\bigg[\sum^{{\tt T}_{\log c}^+-1}_{n=0}\mathrm{e}^{-(\omega_--2\theta){\hat{\tt S}}_n}\bigg]\leq CA^2c^{(2\gamma-\omega_-)}.
\end{equation}

Combining \eqref{eq: proof good I1 final}, \eqref{eq: proof good I2 final} and \eqref{eq: proof good I3 final}, we conclude \cref{lem: estimate good}.
\end{proof}

\bibliographystyle{alphaurl}
\bibliography{biblio}

\end{document}